\documentclass[final,onefignum,onetabnum,dvipsnames]{siamart251216}
\usepackage{amsmath, amssymb, mathtools}
\usepackage[absolute,overlay]{textpos}
\usepackage{graphicx}
\usepackage[nameinlink]{cleveref}
\usepackage{algorithm}
\usepackage{algpseudocode}
\algtext*{EndIf}
\algtext*{EndFor}
\algtext*{EndWhile}
\numberwithin{algorithm}{section}
\usepackage{multirow}
\usepackage{comment}

\numberwithin{equation}{section}
\newsiamremark{remark}{Remark}
\newsiamthm{property}{Property}
\crefname{theorem}{Theorem}{Theorems}
\crefname{corollary}{Corollary}{Corollaries}
\crefname{subsection}{subsection}{subsections}
\Crefname{subsection}{Subsection}{Subsections}
\crefname{algorithm}{Algorithm}{Algorithms}
\crefname{table}{Table}{Tables}
\crefname{property}{Property}{Properties}
\crefname{figure}{Figure}{Figures}
\crefname{appendix}{Appendix}{Appendices}

\newcommand{\colorone}{NavyBlue}
\newcommand{\colortwo}{Green}

\newcommand{\col}[1]{\textcolor{blue!60}{#1}}

\DeclareMathOperator*{\argmin}{\mathrm{argmin}}

\headers{A Jacobi-like algorithm for normal matrices}{Simon Mataigne and P.-A.~Absil}

\title{A Jacobi-like algorithm for normal matrices by the skew-symmetric part\thanks{Submitted to the editors.\funding{Simon Mataigne is a Research Fellow of the Fonds de la Recherche Scientifique - FNRS. This work was supported by the Fonds de la Recherche Scientifique - FNRS under Grant no T.0001.23.}}}

\author{Simon Mataigne\thanks{ICTEAM Institute, UCLouvain, B-1348 Louvain-la-Neuve, Belgium, (\email{simon.mataigne@uclouvain.be}, \email{pa.absil@uclouvain.be}).}
\and P.-A.~Absil\footnotemark[2]
}

\usepackage{amsopn}

\ifpdf
\hypersetup{
  pdftitle={A Jacobi-like algorithm for normal matrices by the skew-symmetric part},
  pdfauthor={Simon Mataigne}
}
\fi

\begin{document}

\maketitle

\begin{abstract}
We present a fast Jacobi-like algorithm for computing the eigenvalues, and optionally the eigenvectors, of a real normal matrix. The method gains a computational advantage by using Paardekooper's method for skew-symmetric matrices The method is most efficient for matrices where most eigenvalues are complex, such as random orthogonal matrices arising in the context of statistics on manifolds. In this case, the method is faster than the other Jacobi-like algorithms. In the last section of this paper, we also give explicit formulas for the nearest symmetric skew-Hamiltonian and the nearest ortho-symplectic matrix. These problems arise in the design and the analysis of the algorithm.
\end{abstract}



\section{Introduction}
Jacobi's algorithm is a famous and one of the oldest algorithms~\cite{jacobi1846leichtes} for solving the symmetric eigenvalue problem (EVP)~\cite{gregory1953computing,GoldstineNeumann59,Rutishauser1966}. Over the years, many variants have been proposed for solving other EVPs, subsumed under the name \emph{Jacobi-like algorithms}. While the operation count of these methods is higher than that of competitors such as the QR algorithm, Jacobi-like algorithms have regained popularity because of their high parallelizability. The shared idea of many Jacobi-like algorithms is to repeatedly solve small-size problems obtained from submatrices of the original matrix. In particular, the classical algorithm for symmetric matrices relies on the explicit solution to the $2\times 2$ symmetric EVP. In 1971, Paardekooper~\cite{Paardekooper1971, Hacon1993} proposed an algorithm for (real) skew-symmetric matrices based on the explicit solution to the $4\times4$ skew-symmetric EVP. This algorithm converges to the real Schur form. For normal matrices, this form is the real-valued analog of the eigenvalue decomposition. Later on, other specialized algorithms were proposed, e.g., for Hamiltonian and skew-Hamiltonian matrices~\cite{FABENDER200137}.

In 1959, a generalization of Jacobi's algorithm was proposed by Goldstine and Horwitz to solve the normal eigenvalue problem~\cite{Goldstine59}. Instead of diagonalizing submatrices, it focuses on decreasing the norm of off-diagonal elements. This efficient algorithm converges quadratically~\cite{Ruhe1967} but relies on complex-valued arithmetic. If the input matrix $A$ is real-valued, relying on complex transformations ignores the additional structure of real matrices and yields complex arithmetic, which is more costly than real arithmetic. 
It is well-known that the symmetric and skew-symmetric parts of a real normal matrix commute. Therefore, the real Schur form can also be computed by the procedure of Bunse-Gerstner et al.~\cite[Algorithm~2]{BunseGerstner1993} for the simultaneous diagonalization of commuting matrices. Moreover, in 2003, an algorithm using only real arithmetic was also proposed by Zhou and Brent~\cite{ZhouBrent2003}. It relies on the assumption that a subroutine can compute the real Schur form of an arbitrary $4\times 4$ matrix. This method is essentially equivalent to~\cite[Algorithm~2]{BunseGerstner1993} and the idea outlined in~\cite[p.~178]{Veselic1979b}.

Finally, other Jacobi-like algorithms, based on non-unitary transformations, were proposed for arbitrary matrices, either real~\cite{Eberlein68} or complex~\cite{Eberlein62,Eberlein70}. For real matrices with complex eigenvalues, the convergence of~\cite{Eberlein68} is slow~\cite{Veselic76}. This lead to the development of the quadratically convergent method of Veseli\'c~\cite{Veselic79} for arbitrary real matrices with complex eigenvalues. Early ideas of taking advantage of the skew-symmetric part for arbitrary matrices can be found in Veseli\'c's method. 

In this paper, we extend the idea of exploiting the skew-symmetric part for real normal matrices, introduced in \cite{mataignegallivan2025}, to develop a Jacobi-like algorithm for the real normal EVP. The method is based on the sequential application of modified versions of Paardekooper’s method, Jacobi’s algorithm for symmetric skew-Hamiltonian matrices~\cite{FABENDER200137}, and related techniques. The fast convergence of Paardekooper's method~\cite{Rhee1995} allows this new method to outperform other Jacobi-like algorithms for normal matrices. We analyze the behavior of the algorithm through a theoretical perturbation analysis building on that of~\cite{mataignegallivan2025}. Finally, we demonstrate the speed, accuracy, and robustness of the proposed method through a set of numerical experiments.

The real normal eigenvalue problem plays a central role in computational applications. It arises in statistical inference and optimization on the special orthogonal group and Stiefel manifolds, where preserving orthogonality and exploiting normal structure are essential for stable and geometry-preserving computations~\cite{ChakrabortyVemuri2019, ZimmermannHuper22, RentmeestersQ, mataigne2024, krakowski07}. It also appears as a fundamental building block in algorithms for generalized eigenvalue problems, where reductions to structured normal forms can improve numerical robustness and efficiency~\cite{CHARLIER1990,BUNSEGERSTNER1991741}. In signal processing, the real normal EVP underlies the computation of Pisarenko frequency estimates, a classical technique for high-resolution spectral estimation and harmonic retrieval~\cite{Ammar1988determination,Cybenko85}.
\paragraph{Organization of the paper} We start by giving an overview of the real normal EVP in \cref{sec:theorems}. In particular, we recall its relation with that of the skew-symmetric part~\cite{mataignegallivan2025}. Then, in \cref{sec:paardekooper}, we recall the functioning of Paardekooper's method for skew-symmetric matrices and propose an \emph{implicit} variant. We illustrate the functioning of our algorithm on a simple instance in \cref{sec:example} and then, in \cref{sec:phases}, we present it in detail. In \cref{sec:floating_arithmetic}, we conduct a perturbation analysis which explains the design and the behavior of the algorithm. Robustness, running time and accuracy are tested in \cref{sec:numerical_experiments}. Finally, relevant results, used in \cref{sec:floating_arithmetic}, about the nearest symmetric skew-Hamiltonian matrix and the nearest ortho-symplectic matrix are given in \cref{sec:SSH_and_os}.


\paragraph{Reproducibility} An implementation of the method, as well as the codes to reproduce all experiments presented in this paper, are available at the address \url{https://github.com/smataigne/NormalJacobi.jl}.
\paragraph{Notation} $I_n$ denotes the $n\times n$ identity matrix. The orthogonal group and the special orthogonal group of $n\times n$ matrices are denoted, respectively, by $\mathrm{O}(n)$ and $\mathrm{SO}(n)$. For $n\geq k$, the Stiefel manifold of orthonormal $k$-frames in $\mathbb{R}^n$ is denoted by $\mathrm{St}(n,k)\coloneq \{V\in\mathbb{R}^{n\times k}\ | \ V^\top V = I_k\}$. $A\otimes B$ refers to the Kronecker product of two matrices $A$ and $B$. $\mathrm{Sym}(n)$ and $\mathrm{Skew}(n)$ denote, respectively, the sets of symmetric and skew-symmetric matrices of size $n$. We use $\mathrm{sym}(A) = \frac{1}{2}(A+A^\top)$ and $\mathrm{skew}(A) = \frac12(A-A^\top)$ for, respectively, the symmetric and the skew-symmetric part of the matrix~$A$. Finally, an even-odd permutation matrix $P_{\mathrm{eo}}\in\mathrm{O}(n)$ is a permutation matrix such that for every vector $[a_1\ a_2\ ...\  a_n]\in\mathbb{R}^{1\times n}$, we have 
\begin{equation*}
	[a_1\ a_2\ ...\  a_n]P_{\mathrm{eo}} = [a_1\ a_3\ ...\ a_{\lceil\frac{n}{2}\rceil}\ a_2\ a_4\ ... \ a_{\lfloor\frac{n}{2}\rfloor}]. 
\end{equation*}
\section{The normal eigenvalue problem by the skew-symmetric part}\label{sec:theorems}
The normal matrices are the matrices that can be diagonalized under complex unitary similarity transformation. However, if we consider \emph{real} matrices only, it is natural to restrict the similarity transformation to \emph{orthogonal} matrices, and thus real arithmetic. Then, the corresponding canonical form is the real Schur form, i.e., a block diagonal matrix with $1\times 1$ and/or $2\times 2$ diagonal blocks. The $2\times 2$ blocks have the form of~\eqref{eq:2x2block}, which is easily related to its eigenvalues:
\begin{equation}\label{eq:2x2block}
	\lambda \begin{bmatrix}
	\cos(\theta)&-\sin(\theta)\\
	\sin(\theta)&\cos(\theta)
	\end{bmatrix} = \frac{1}{2}\begin{bmatrix}
	1&1\\
	i&-i
	\end{bmatrix}
	\begin{bmatrix}
	\lambda e^{-i\theta}&0\\
	0&\lambda e^{i\theta}
	\end{bmatrix}
	\begin{bmatrix}
	1&-i\\
	1&i
	\end{bmatrix}.
\end{equation}
The real Schur form is sometimes called Murnaghan's form~\cite{MurnaghanWintner1931}. For normal matrices, the following theorem holds.
\begin{theorem}[{\cite[Thm.~2.5.8]{HornJohnson2013}}]\label{prop:normalEVP}
	A matrix $A\in\mathbb{R}^{n\times n}$ is normal if and only if it admits a \emph{real Schur decomposition (RSD)} by an orthogonal transformation $Q\in\mathrm{O}(n)$, i.e.,
	\begin{equation*}
		A = QSQ^\top,
	\end{equation*}
	where $S\in \mathbb{R}^{n\times n}$ is block diagonal with $r\geq 0$ blocks of size $1\times 1$ and $p\coloneq \frac{n - r}{2}$ blocks of size $2\times 2$ of the form~\eqref{eq:2x2block}. The columns of $Q$ are termed \emph{Schur vectors}. W.l.o.g., the first subdiagonal of $S$ is imposed nonnegative.
\end{theorem}

Letting $I_2=\left[\begin{smallmatrix}1&0\\ 0 &1\end{smallmatrix}\right] $ and $J_2=\left[\begin{smallmatrix}0&-1\\
1&0\end{smallmatrix}\right]$, there are positive diagonal matrices $\Lambda,\Theta\in\mathbb{R}^{p\times p}$ and a diagonal matrix $\breve{\Lambda}\in \mathbb{R}^{r\times r}$ with $2p+r=n$ such that 
\begin{equation}
	S = \begin{bmatrix}
	\Lambda\cos(\Theta) \otimes I_2+ \Lambda\sin(\Theta)\otimes J_2&0\\
	0&\breve{\Lambda}
	\end{bmatrix}.
\end{equation}
In~\cite[Alg.~4.1]{mataignegallivan2025} a QR-type algorithm is proposed for computing the RSD of a normal matrix $A$. In particular, the method gets a computational advantage by computing the RSD of the skew-symmetric part $$\Omega \coloneq \mathrm{skew}(A).$$ In the present paper, we further exploit the properties of the skew-symmetric part to design a Jacobi-like algorithm. Let us recall relevant results from~\cite{mataignegallivan2025}.


\subsection{Eigenvalues with distinct nonzero imaginary parts} Consider the real Schur decomposition of $\Omega = \mathrm{skew}(A)$:
\begin{equation*}
	\Omega  = Q_\Omega\begin{bmatrix}
	\Lambda\sin(\Theta)\otimes J_2&0\\
	0&0_r
	\end{bmatrix}Q_\Omega^\top.
\end{equation*} 
Let $V\in\mathrm{St}(n,2k)$ gather the $2k$ columns of $Q_\Omega$ corresponding to the $2\times 2$ blocks of $\Lambda\sin(\Theta)\otimes J_2$ with simple (i.e., non-repeated) eigenvalues of $\Omega$. Let $\Lambda_k \sin(\Theta_k)\otimes J_2$ be the block-diagonal matrix containing these $2\times 2$ blocks. Then, 
$$\Omega V = V (\Lambda_k \sin(\Theta_k)\otimes J_2),$$
and by~\cite[Theorem~3.3]{mataignegallivan2025},
\begin{align}
	\label{eq:evp_complex_distinct}
	AV = V (\Lambda_k\cos(\Theta_k) \otimes I_2+ \Lambda_k\sin(\Theta_k)\otimes J_2).
\end{align}
A consequence of \eqref{eq:evp_complex_distinct} is that the matrix $V^\top A V$ is block-diagonal and $V^\top A V$ is a submatrix of the Schur form $S$ of $A$.

\subsection{Eigenvalues with repeated nonzero imaginary parts}\label{sec:complex_repeated} Now, assume that exactly $m>1$ eigenvalues have the repeated imaginary part $\sigma>0$. Then there is a permuted subset $V\in\mathrm{St}(n, 2m)$ of the columns of $Q_\Omega$ verifying
\begin{equation*}
	\Omega V = V (J_2\otimes \sigma I_{m}) = V\begin{bmatrix}
	0&-\sigma I_m\\
	\sigma I_m&0
	\end{bmatrix},
\end{equation*}
and, by \cite[Thm.~3.4]{mataignegallivan2025}, there are $\widetilde{H}\in\mathrm{Sym}(m)$ and $\widetilde{\Omega}\in\mathrm{Skew}(m)$ such that 
\begin{equation}\label{eq:evp_complex_repeated}
	AV = V (I_2 \otimes \widetilde{H}+ J_2 \otimes (\widetilde{\Omega} + \sigma I_m))=V\begin{bmatrix}
	\widetilde{H}&-\widetilde{\Omega}-\sigma I_m\\
	\widetilde{\Omega}+\sigma I_m&\widetilde{H}
	\end{bmatrix}.
\end{equation}
The matrix $M=I_2 \otimes \widetilde{H}+ J_2 \otimes \widetilde{\Omega}=\left[\begin{smallmatrix}\widetilde{H}&-\widetilde{\Omega}\\
	\widetilde{\Omega}&\widetilde{H}\end{smallmatrix}\right] $ belongs to the set $\mathrm{SSkH}(2m)$ of $2m\times 2m$ \emph{symmetric skew-Hamiltonian (SSkH)} matrices \cite{FABENDER200137}:
\begin{equation*}
	\mathrm{SSkH}(2m)\coloneq \left\{\begin{bmatrix}
	\widetilde{H}&-\widetilde{\Omega}\\
	\widetilde{\Omega}&\widetilde{H}
	\end{bmatrix}\ | \ \widetilde{H}\in\mathrm{Sym}(m),\ \widetilde{\Omega}\in\mathrm{Skew}(m)\right\}.
\end{equation*} 
The eigenvalue decomposition of the matrix~$M\in\mathrm{SSkH}(2m)$ admits a form (see, e.g., \cite[Lem.~B.1]{mataignegallivan2025}) \begin{equation*}
	M=R(I_2\otimes D_m) R^\top = R\begin{bmatrix}
	D_m&0\\
	0&D_m
	\end{bmatrix}R^\top,
\end{equation*} 
where $D_m$ is diagonal and $R$ belongs to the set $\mathrm{OSp}(2m)$ of $2m\times 2m$ \emph{ortho-symplectic} matrices \cite[Eq.~5.1]{dopico09}, i.e., $R$ is orthogonal and commutes with $J_2\otimes I_m$. This property yields the following structure:
\begin{equation*}
	\mathrm{OSp}(2m)= \left\{\begin{bmatrix}
	U_\mathrm{r}&-U_\mathrm{i}\\
	U_\mathrm{i}&U_\mathrm{r}
	\end{bmatrix}\in\mathrm{O}(2m)\right\}.
\end{equation*} Moreover, the commutativity of $R$ and $J_2\otimes I_m$ yields the commutativity of $R$ and $\sigma J_2\otimes I_m$, and thus
\begin{equation*}
 A (V R)= (VR) R^\top (M+\sigma J_2\otimes I_m)R = (VR)\begin{bmatrix}
	D_m&-\sigma I_m\\
	\sigma I_m&D_m
	\end{bmatrix}.
\end{equation*}
Using an even-odd permutation $P_\mathrm{eo}$, we obtain the identity
\begin{equation*}
	P_\mathrm{eo}\begin{bmatrix}
	D_m&-\sigma I_m\\
	\sigma I_m&D_m
	\end{bmatrix}P_\mathrm{eo}^\top  = D_m\otimes I_2 + \sigma I_m\otimes J_2.
\end{equation*}
Since $D_m\otimes I_2 + \sigma I_m\otimes J_2$ is block diagonal, we conclude that it is a submatrix of the Schur form $S$ of $A$.

\subsection{Real eigenvalues} Finally, a similar result can be obtained for the $r$ real eigenvalues of $A$. Assume the matrix $V\in\mathrm{St}(n, r)$ gathers all columns of $Q_\Omega$ spanning the nullspace of $\Omega$, i.e.,
\begin{equation*}
	\Omega V = 0_r.
\end{equation*}
Then, by \cite[Thm.~3.5]{mataignegallivan2025}, there is a symmetric matrix $H\in\mathrm{Sym}(r)$ such that
\begin{equation}\label{eq:evp_real}
	AV = V H.
\end{equation}
By \eqref{eq:evp_real}, if $V$ is known and the real eigenvectors of $A$ are desired, computing the eigenvalue decomposition $H = \breve{R}\breve{\Lambda}\breve{R}^\top$ yields $A(V\breve{R}) = (V\breve{R})\breve{\Lambda}$.


\section{Paardekooper's algorithm for skew-symmetric matrices}\label{sec:paardekooper}
The starting point of the method proposed in this paper is computing the RSD of the skew-symmetric part $\Omega$ of the real normal matrix $A$.
The Jacobi-like algorithm for solving this problem is called Paardekooper's method for skew-symmetric matrices~\cite{Paardekooper1971}. Since this method is not commonly known, this section recalls the essential ideas of Paardekooper's method. It converges globally~\cite{Paardekooper1971}, and locally at least quadratically~\cite{Rhee1995} to the real Schur form.
\subsection{The $2\times 2$ Singular Value Decomposition}
Paardekooper's method relies on the explicit solution to the $2\times 2$ Singular Value Decomposition (SVD). Given a $2\times 2$ real matrix $\left[\begin{smallmatrix}a_{11}&a_{12}\\
	a_{21}&a_{22}\end{smallmatrix}\right]$, there is a pair of Givens rotations with angles $\alpha_1,\alpha_2\in\mathbb{R}$ such that
\begin{equation}\label{eq:2x2svd}
	\begin{bmatrix}
	\cos(\alpha_1)&\sin(\alpha_1)\\
	-\sin(\alpha_1)&\cos(\alpha_1)
	\end{bmatrix}
	\begin{bmatrix}
	a_{11}&a_{12}\\
	a_{21}&a_{22}
	\end{bmatrix}\begin{bmatrix}
	\cos(\alpha_2)&-\sin(\alpha_2)\\
	\sin(\alpha_2)&\cos(\alpha_2)
	\end{bmatrix} =\begin{bmatrix}
	d_1&0\\
	0& d_2
	\end{bmatrix},
\end{equation}
with $d_1,d_2\in\mathbb{R}$. Details on how to compute the angles $\alpha_1$ and $\alpha_2$ can be found, e.g., in~\cite[Section~1]{Paardekooper1971}.
\begin{remark}
The transformation in~\eqref{eq:2x2svd} uses a pair of \emph{rotations} (i.e., as opposed to \emph{reflections}) which does not allow enforcing the nonnegativity of $d_1$ and $d_2$.
\end{remark}
We use the following notation for Givens rotations of size larger than $2\times 2$. For integers $1\leq i\leq j\leq n$ and an angle $\alpha$, the Givens rotation $ G(\alpha, i, j)\in\mathrm{SO}(n)$ is
\begin{equation*}
	G(\alpha, i, j)_{k,l} = \begin{cases}
	1\quad &\text{if}\quad k = l \notin \{i, j\},\\
	\cos(\alpha)\quad &\text{if}\quad k = l \in\{i,j\},\\
	\sin(\alpha)\quad &\text{if}\quad  k = j\ \text{and}\ l = i,\\
	-\sin(\alpha)\quad &\text{if}\quad k = i\ \text{and}\ l = j,\\
	0\quad  &\text{ortherwise}.
	\end{cases}
\end{equation*}
\subsection{Paardekooper's step for a $4\times 4$ skew-symmetric submatrix}\label{sec:4x4skew}
Based on the decomposition \eqref{eq:2x2svd}, Paardekooper~\cite{Paardekooper1971} computes the real Schur form of a $4\times 4$ skew-symmetric matrix $\Omega$ in finitely many arithmetic operations. First, it requires computing rotations from \eqref{eq:2x2svd} for the submatrix $\left[\begin{smallmatrix}\omega_{21}&-\omega_{32}\\ \omega_{41}&\omega_{43}\end{smallmatrix}\right]$:
{\small
\begin{align*}
	G(\alpha_1,2,4)^\top
	\underbrace{\begin{bmatrix}
		0&-\omega_{21}&-\omega_{31}&-\omega_{41}\\
		\col{\omega_{21}}&0& \col{- \omega_{32}}&-\omega_{42}\\
		\omega_{31}&\omega_{32}&0&\omega_{43}\\
		\col{\omega_{41}}&\omega_{42}&\col{\omega_{43}}&0\\
	\end{bmatrix}}_{=\Omega}G(\alpha_2,1,3)
	= \begin{bmatrix}
		0&-\omega_{21}&-\omega_{31}'&-\omega_{41}\\
		\col{d_1}&0& \col{0}&-\omega_{42}'\\
		\omega_{31}'&\omega_{32}&0&\omega_{43}\\
		\col{0}&\omega_{42}'&\col{d_2}&0\\
	\end{bmatrix}.
\end{align*}}Notice that $G(\alpha_1,2,4)$ and $G(\alpha_2,1,3)$ commute and, hence, 
the eigenvalues are restored by completing the similarity transformation with
{\small
\begin{align*}
G(\alpha_2,1,3)^\top
	\begin{bmatrix}
		0&\col{-\omega_{21}}&-\omega_{31}'&\col{-\omega_{41}}\\
		d_1&0& 0&-\omega_{42}'\\
		\omega_{31}'&\col{\omega_{32}}&0&\col{\omega_{43}}\\
		0&\omega_{42}'&d_2&0\\
	\end{bmatrix}G(\alpha_1,2,4)
	= \underbrace{\begin{bmatrix}
		0&\col{-d_1}&-\omega_{31}''&\col{0}\\
		d_1&0& 0&-\omega_{42}''\\
		\omega_{31}''&\col{0}&0&\col{-d_2}\\
		0&\omega_{42}''&d_2&0\\
	\end{bmatrix}}_{\eqcolon \Omega^{(1)}}.
\end{align*}}Zeros have been introduced on the anti-diagonal of $\Omega^{(1)}$. A second pair of rotations gives the real Schur form of $\Omega$. The angles $\widehat{\alpha}_1,\widehat{\alpha}_2$ are obtained from the submatrix $\left[\begin{smallmatrix}d_1&-\omega_{42}''\\ \omega_{31}''&-d_2\end{smallmatrix}\right]$ of $\Omega^{(1)}$. Let $G\coloneq G(\widehat{\alpha}_1,2,3)G(\widehat{\alpha}_2,1,4)$.
Readily applying the full similarity transformation to $\Omega^{(1)}$ yields
{\small
	\begin{align*}
	G^\top
	\begin{bmatrix}
		0&-d_1&-\omega_{31}''&0\\
		\col{d_1}&0& 0&\col{-\omega_{42}''}\\
		\col{\omega_{31}''}&0&0&\col{-d_2}\\
		0&\omega_{42}''&d_2&0\\
	\end{bmatrix} G= \underbrace{\begin{bmatrix}
		0&-d_1'&0&0\\
		\col{d_1'}&0& 0&\col{0}\\
		\col{0}&0&0&\col{d_2'}\\
		0&0&-d_2'&0\\
	\end{bmatrix}}_{\eqcolon \Omega^{(2)}}.
\end{align*}}Finally, the similarity transformation by $D \coloneq \mathrm{diag}(1,\mathrm{sign}(d_1'),  1, -\mathrm{sign}(d_2'))$ enforces the non-negativity of the subdiagonal:
{\small
\begin{equation*}
	D\begin{bmatrix}
		0&-d_1'&0&0\\
		d_1'&0& 0&0\\
		0&0&0&d_2'\\
		0&0&-d_2'&0\\
	\end{bmatrix}D=\begin{bmatrix}
		0&-\sigma_1&0&0\\
		\sigma_1&0& 0&0\\
		0&0&0&-\sigma_2\\
		0&0&\sigma_2&0\\
	\end{bmatrix}, \quad \text{with}\ \sigma_1,\sigma_2\geq 0.
\end{equation*}}This concludes Paardekooper's solution to the $4\times 4$ skew-symmetric EVP. Observe that $\sigma_1$ and $\sigma_2$ are the singular values of $\Omega$, $\pm i \sigma_1$ and $\pm i\sigma_2$ are the eigenvalues of $\Omega$. 
\begin{remark}
If $n$ is odd, then either the $3\times 3$ skew-symmetric EVP must be solved (see~\cref{sec:3x3skew}), or a row and a column of zeros can be appended to the matrix to make $n$ even.
\end{remark}

\subsection{Sweeping strategy and pseudo-code}
All Jacobi-like algorithms require a \emph{sweeping strategy}, i.e., an order in which submatrices are diagonalized or block-diagonalized. Sweeping strategies have been extensively studied, in particular for parallel implementations, see, e.g.~\cite{Sameh1971,Brent85,LukPark1989,Shroff1990}. The standard and cheapest sweeping strategy is the \emph{cyclic strategy}~\cite{Forsythe1960TheCJ,Wilkinson1962}, where submatrices are selected in a cyclic order. This strategy is used in the pseudo-codes of this paper.

Finally, convergence is defined by a stopping criterion. Paardekooper’s method enforces convergence toward a block-diagonal matrix with (possibly zero) $2\times 2$ diagonal blocks. In this manuscript, the function ``$\mathrm{offschur}$'' is defined by
\begin{equation*}
	\mathrm{offschur}:\mathbb{R}^{n\times n}\rightarrow\mathbb{R}^+:A \mapsto \sqrt{\sum_{\substack{i=1\\i\ \text{odd}}}^{n-2}\sum_{j=i+2}^n a_{i,j}^2+a_{i+1,j}^2  + a_{j,i}^2+a_{j,i+1}^2}.
\end{equation*}
\cref{alg:paardekooper} gives the pseudo-code of the cyclic Paardekooper method for skew-symmetric matrices. Operations specific to Paardekooper’s method are written in \textcolor{\colorone}{blue} and marked with (i). In this paper, however, our goal is to obtain the RSD of a general normal matrix $A$. Therefore, we use a variant of this algorithm computing the transformations from the skew-symmetric part $\Omega\coloneq \mathrm{skew}(A)$ but applying these transformations to $A$. The sweeps are \emph{implicitly} applied on $\Omega$ by linearity since $\mathrm{skew}(Q^\top AQ) = Q^\top \mathrm{skew}(A)Q$. We term this variant the \emph{Implicit Paardekooper's method}. The pseudo-code of \cref{alg:paardekooper} includes this variant in \textcolor{\colortwo}{green} and marked with (ii).
\begin{algorithm}[t]
	\caption{\textcolor{\colorone}{(i) Cyclic Paardekooper's method for skew-symmetric matrices}\\
	\hspace*{2.58cm}\textcolor{\colortwo}{(ii) Implicit cyclic Paardekooper's method}}
	\begin{algorithmic}
		\State \textbf{Input:} \textcolor{\colorone}{(i) $\Omega\in\mathrm{Skew}(n)$} or \textcolor{\colortwo}{(ii) $A\in\mathbb{R}^{n\times n}$} with $n$ even, $Q=I_n$ and $\rho>0$.
		\State \textbf{Output:} $S_\Omega \in\mathbb{R}^{n\times n}$ and $Q\in\mathrm{O}(n)$.
		\While{ \textcolor{\colorone}{(i) $\mathrm{offschur}(\Omega)> \rho \|\Omega\|_\mathrm{F}$} or \textcolor{\colortwo}{(ii) $\mathrm{offschur}(\mathrm{skew}(A))> \rho \|A\|_\mathrm{F}$}}
			\For{$i = 1:2:n-3$}
				\For{$j = i+2:2:n-1$}
					\State Set $\mathtt{l}=\{i,i+1,j,j+1\}$.
					\State \textcolor{\colortwo}{(ii) Compute $\Omega_{\mathtt{l},\mathtt{l}} = \frac12(A_{\mathtt{l},\mathtt{l}} -A_{\mathtt{l},\mathtt{l}}^\top) $.}
					\State Compute the real Schur decomposition $G^\top\Omega_{l\mathtt{l},\mathtt{l}}G=S$ by \cref{sec:4x4skew}.
					\State \textcolor{\colorone}{(i) Update $\Omega_{\mathtt{l},:} \leftarrow  G^\top \Omega_{\mathtt{l},:}$} or \textcolor{\colortwo}{(ii) Update $A_{\mathtt{l},:} \leftarrow  G^\top A_{\mathtt{l},:}$}.
					\State \textcolor{\colorone}{(i) Update $\Omega_{:,\mathtt{l}} \leftarrow \Omega_{:, \mathtt{l}}G$} or   \textcolor{\colortwo}{(ii) Update $A_{:,\mathtt{l}} \leftarrow A_{:, \mathtt{l}}G$}.
					\State Update $Q_{:,\mathtt{l}} \leftarrow Q_{:, \mathtt{l}}G$.
				\EndFor
			\EndFor
			\EndWhile
		\Return \textcolor{\colorone}{(i) $S_\Omega\coloneq \Omega$} or  \textcolor{\colortwo}{(ii) $S_\Omega\coloneq A$} and $Q$.
	\end{algorithmic}
	\label{alg:paardekooper}
\end{algorithm}
\section{Normal matrices: a simple $4\times 4$ example}\label{sec:example} Before giving the details of our method (\cref{alg:normalpaardekooper}), we illustrate its functioning on a simple $4\times 4$ instance. 
Consider the real normal matrix $A$ defined by
\begin{equation*}
	A = \begin{bmatrix}
	1 &1&1&-1\\
	1&1&-1&1\\
	1&-1&-1&-1\\
	1&-1&1&1
	\end{bmatrix},\quad\text{with}\ \Omega= \begin{bmatrix}
	0&0&0&-1\\
	0&0&0&1\\
	0&0&0&-1\\
	1&-1&1&0
	\end{bmatrix}.
\end{equation*}
It can be verified that the four eigenvalues of the matrix $A$ are $\{\pm 2,1\pm i\sqrt{3}\}$. 

\textbf{Step I} of \cref{alg:normalpaardekooper} applies implicitly Paardekooper's method on $A$ until $\Omega$ is block-diagonalized. It can be verified that Paardekooper's method gives the $4\times 4$ rotation
\begin{equation*}
	G = \begin{bmatrix}
	-\frac{\sqrt{2}}{2}&0&\frac{\sqrt{2}}{2}&0\\
	0&0&0&1\\
	-\frac{\sqrt{2}}{2}&0&-\frac{\sqrt{2}}{2}&0\\
	0&-1&0&0
	\end{bmatrix} \begin{bmatrix}
	\frac{\sqrt{6}}{3}&0&0&-\frac{\sqrt{3}}{3}\\
	0&1&0&0\\
	0&0&1&0\\
	\frac{\sqrt{3}}{3}&0&0&\frac{\sqrt{6}}{3}
	\end{bmatrix}.
\end{equation*}
By applying the similarity transformation on the matrix $A$, we obtain the update $A^{(1)}= G^\top A G$ such that
\begin{equation*}
A^{(1)}  = \begin{bmatrix}
	1 &-\sqrt{3}&0&0\\
	\sqrt{3}&1&0&0\\
	0&0&-1&\sqrt{3}\\
	0&0&\sqrt{3}&1
	\end{bmatrix},\quad\text{and}\ \Omega^{(1)}= \begin{bmatrix}
	0&-\sqrt{3}&0&0\\
	\sqrt{3}&0&0&0\\
	0&0&0&0\\
	0&0&0&0
	\end{bmatrix}.
\end{equation*}
Since $n=4$, one sweep completes step I. The eigenvalues $\pm i\sqrt{3}$ of $ \Omega^{(1)}$ are not repeated. Hence, the corresponding block of $A^{(1)}$ gives the eigenvalues $1\pm i\sqrt{3}$ of $A$ directly. The eigenvalue zero of $\Omega^{(1)}$ has multiplicity~$2$, positioned at the diagonal indices $\mathtt{l}_0=\{3,4\}$. Therefore, we isolate the  $2\times 2$ symmetric matrix $A_{\mathtt{l}_0,\mathtt{l}_0}$ containing the real eigenvalues of $A$. 

\textbf{Step II} gives the real eigenvalues of the matrix $A_{\mathtt{l}_0,\mathtt{l}_0}$ by a single Jacobi rotation
\begin{equation*}
\begin{bmatrix}
\frac12 & \frac{\sqrt{3}}{2}\\
-\frac{\sqrt{3}}{2}&\frac12
\end{bmatrix} \begin{bmatrix}
-1&\sqrt{3}\\
\sqrt{3}&1
\end{bmatrix} \begin{bmatrix}
\frac12 & -\frac{\sqrt{3}}{2}\\
\frac{\sqrt{3}}{2}&\frac12
\end{bmatrix} = \begin{bmatrix}
2&0\\
0&-2
\end{bmatrix}.
\end{equation*}
We recover the real eigenvalues $\pm 2$ of the matrix $A$.

\textbf{Step III} only applies with inexact Schur decompositions and/or inexact arithmetic.

We have thus diagonalized the normal matrix $A$ with a few arithmetic operations. In the next section, we generalize the procedure for arbitrary dimensions.
\section{The algorithm}\label{sec:phases}
This section describes~\cref{alg:normalpaardekooper}, a new Jacobi-like algorithm for computing the real Schur form of a real normal matrix $A$. It is divided in three steps identifying the subproblems of~\cref{sec:theorems}, and new subproblems described by the perturbation analysis conducted in~\cref{sec:floating_arithmetic}. The result of a numerical experiment showing the output of each step is illustrated in \cref{fig:four_phases}.
\begin{algorithm}[t]
	\begin{algorithmic}
		\State \textbf{Input:} $A\in\mathbb{R}^{n\times n}$ normal with $n$ even, $Q = I_n$, $\rho>0$.
		\State \textbf{Output:} $S\in\mathbb{R}^{n\times n}$ and $Q\in\mathrm{O}(n)$ such that $\mathrm{offschur}(S)\leq \rho \|A\|_\mathrm{F}$ and $A=QSQ^\top$.
		\State \textbf{step I.1:} Apply \cref{alg:paardekooper} with $(A, Q, \rho)$ as input. The output satisfies $\mathrm{offschur}(\mathrm{skew}(A))\leq \rho \|A\|_\mathrm{F}$.
		\State \textbf{step I.2:} Build the adjacency matrix $A_{\mathrm{adj}}$ with $(A,\rho)$ as input of \cref{alg:Adjacency}.
		\State \textbf{step I.3:} Find the indices $l_i$ of each connected component of $A_{\mathrm{adj}}$ for $i=1,2,...$
		\For{every connected component with indices $l_i$}
			\If{$|l_i|\geq 4$ and $\| \mathrm{offschur}(A_{l_i,l_i}- \mathrm{sskh2}(A_{l_i,l_i}))\|_\mathrm{F} < \sqrt{\rho \|A\|_\mathrm{F}}$ (see \eqref{eq:sskh2})}
				\State  \textbf{step II.1:} Apply \cref{alg:ssh_jacobi} with $(A, Q, l_i, \rho)$ as input. The output satisfies $\mathrm{offdiag}(\mathrm{sskh2}(A_{l_i,l_i}))\leq\rho \|A\|_\mathrm{F}$.
			\ElsIf{$\| \mathrm{skew}(A_{l_i,l_i})\|_\mathrm{F} < \sqrt{\rho \|A\|_\mathrm{F}}$}
				\State \textbf{step II.2:} Apply \cref{alg:implicit_symmetric_jacobi} with  $(A, Q, l_i, \rho)$ as input. The output satisfies $\mathrm{offdiag}(\mathrm{sym}(A_{l_i,l_i}))\leq \rho \|A\|_\mathrm{F}$.
			\Else
				\State  \textbf{step II.3:} Apply~\cref{alg:Zhou} with  $(A, Q, l_i, \sqrt{\rho})$ as input.  Break~iterations if a sweep increases $\mathrm{offschur}(A)$ or if the number of sweeps exceeds $5|l_i|$.
			\EndIf
		\EndFor
		\If{$\mathrm{offschur}(A) > \rho \|A\|_\mathrm{F}$}
			\State \textbf{step III:} Apply \cref{alg:Zhou} with $(A, Q, l=\{1,...,n\},\rho)$ as input.
		\EndIf
		\Return $S\coloneq A$ and $Q$.
	\end{algorithmic}
	\caption{(In-place) Jacobi-like algorithm for real normal matrices}
	\label{alg:normalpaardekooper}
\end{algorithm} 
\begin{itemize}
	\item After step I, the skew-symmetric part is block-diagonalized within a prescribed accuracy $\rho>0$: $\mathrm{offschur}(\mathrm{skew}(A))\leq\rho\|A\|_\mathrm{F}$. The normal matrix $A$ is divided in isolated non-zero blocks. These blocks are automatically identified.
	\item After step II, the normal matrix $A$ is ``roughly'' block-diagonalized, i.e., it satisfies $\mathrm{offschur}(A)\lesssim \sqrt{\rho\|A\|_\mathrm{F}}$. This step is itself divided in substeps, described later in details. The justification for using $\sqrt{\rho\|A\|_\mathrm{F}}$ is the use of a quadratically convergent method at step III.
	\item At step III, a final, computationally cheap, accuracy improvement step block-diagonalizes $A$ within prescribed accuracy: $\mathrm{offschur}(A)\leq\rho\|A\|_\mathrm{F}$.
\end{itemize}
Step III requires another Jacobi-like algorithm for normal matrices (\cref{alg:Zhou}). In theory, it can be any of the methods from~\cite{BunseGerstner1993,ZhouBrent2003,Veselic1979b}, or complex-valued methods such as~\cite{Goldstine59,Eberlein70}. In our practical implementation, we use the algorithm of~\cite{ZhouBrent2003}. For step III to be computationally cheap, $\rho$ should be small enough, e.g., the unit roundoff, to benefit from the quadratic local convergence of \cref{alg:Zhou}. In our experiments, we have observed that this property ensures step~III only requires at most two sweeps (most often one) to achieve the target accuracy $\rho$.

\begin{figure}[ht]
	\centering
	\includegraphics[width = \textwidth]{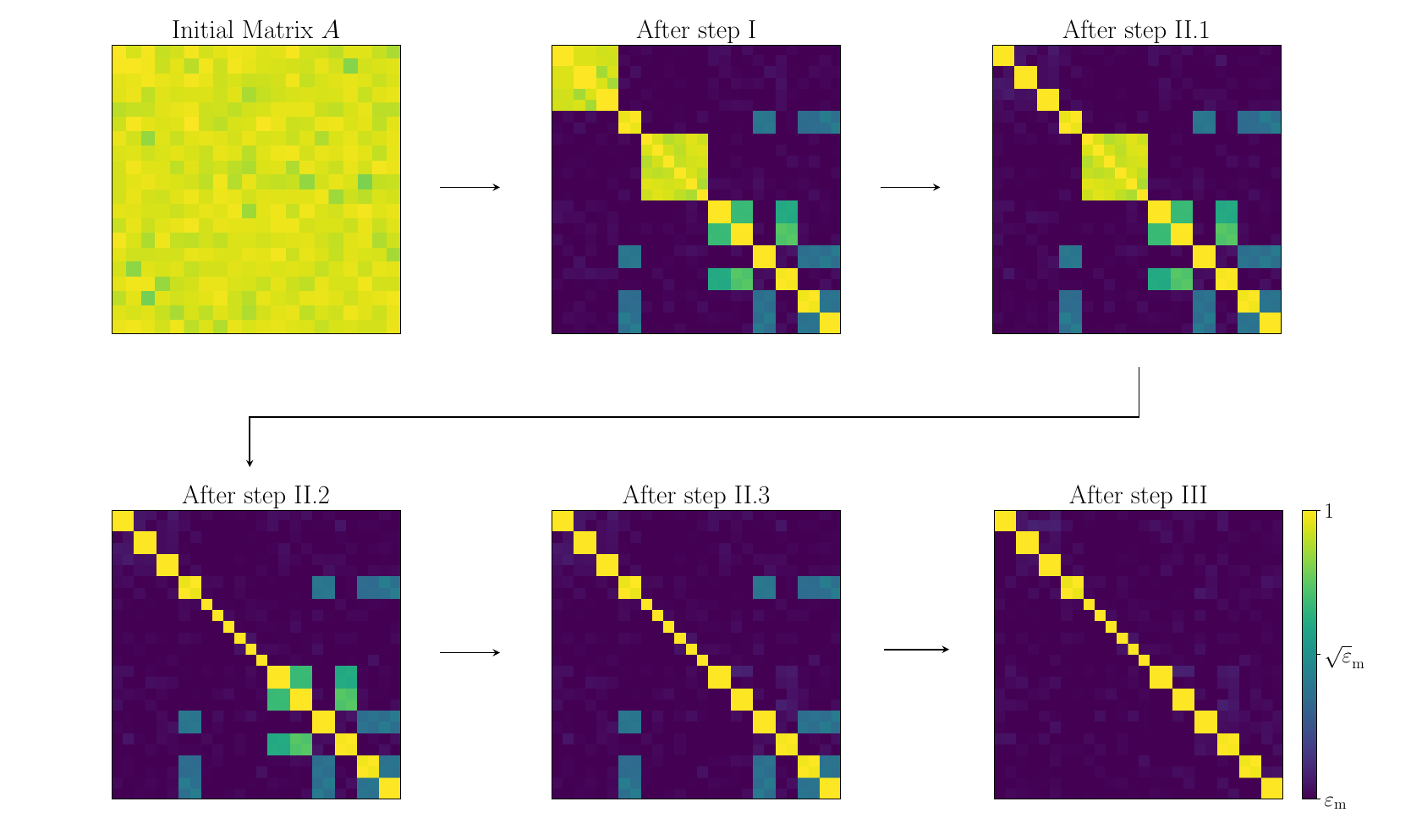}
	\vspace*{-0.6cm}
	\caption{Vizualisation of the steps of \cref{alg:normalpaardekooper} for a numerical experiment with a $26\times 26$ random normal matrix. The matrix is build for triggering all steps: there are 3 pairs of eigenvalues with repeated nonzero imaginary parts for step II.1, 6 real eigenvalues for step II.2, 3 pairs of eigenvalues with close nonzero imaginary parts for step II.3 and, finally, 3 pairs of eigenvalues with imaginary parts separated enough to avoid both step II.1 and step II.3, but close enough to trigger step III. The heatmaps represent the modulus of the matrix elements in logarithmic scale. Yellow is $\mathcal{O}(1)$, green is $\mathcal{O}(\sqrt{\varepsilon_\mathrm{m}})$ and dark blue is $\mathcal{O}(\varepsilon_\mathrm{m})$.}
	\label{fig:four_phases}
	\vspace*{-0.3cm}
\end{figure}
\begin{remark}
In the pseudo-code of \cref{alg:normalpaardekooper} and in this paper, we assume that the matrix size $n$ is even. This assumption is used for symplifying the pseudo-codes but nothing prevents $n$ being odd in practice. 
\end{remark}
\subsection{Step I.1: Implicit Paardekooper's method}\label{sec:phase_I}
\cref{alg:normalpaardekooper} starts by applying Paardekooper's method of \cref{alg:paardekooper} implicitly on $A$ with the stopping criterion $\rho$. Since Paardekooper's sweeps are computed from the entries of~$\Omega\coloneq \mathrm{skew}(A)$, $\Omega$ converges to its Schur form. Assume that after $K$ sweeps, it holds that $\mathrm{offschur}(\Omega) < \rho \|A\|_\mathrm{F}$. Then \cref{alg:paardekooper} stops. We denote this part of the algorithm by \emph{step I.1}. 

\subsection{Steps I.2 and I.3: Detection of unresolved blocks} As illlustrated in \cref{fig:four_phases}, and partly explained by~\cref{sec:theorems}, some off-diagonal blocks of $A$ may not be near zero. For example, by \eqref{eq:evp_real}, if $A$ admits real eigenvalues, a symmetric submatrix remains. By \eqref{eq:evp_complex_repeated}, if $A$ admits eigenvalues with repeated nonzero imaginary parts, submatrices should be symmetric skew-Hamiltonian. Finally, another type of non-near-zero submatrices results from the stopping criterion $\rho$ in Paardekooper's method. The theoretical explanation is given in \cref{sec:floating_arithmetic}.

Step I.2 finds the nonzero submatrices of $A$. To this end, $A$ can be interpreted as the weighted adjacency matrix of a directed graph. We define a matrix $A_{\mathrm{adj}}$ whose entries are either $1$ or $0$ if, respectively, the absolute value of the corresponding entry of $A$ exceeds a prescribed threshold or not. We consider the threshold $\sqrt{\rho\|A\|_\mathrm{F}}$ to ensure that $\mathrm{offschur}(A)\lesssim \sqrt{\rho\|A\|_\mathrm{F}}$ at the end of step II and benefit in practice from the quadratic convergence at step III. By construction, identifying a nonzero submatrix turns into finding a connected component in the graph represented by $A_{\mathrm{adj}}$.

Finding connected components in $A_{\mathrm{adj}}$ is a standard graph problem that can be solved, for example, using Breadth-First Search (BFS) or Depth-First Search (DFS), see, e.g.,~\cite[Chap.~22.2~and~22.3]{CormenLeisersonRivestStein2009}. In BFS, one starts from a node $i$, adds all its neighboring nodes to the current component, and then iteratively includes all unvisited neighbors of those nodes. This procedure continues until no unvisited neighboring node remains, at which point the entire connected component has been identified.

The $2\times 2$ block structure of the Schur form $S$ imposed by Paardekooper's method. Therefore, for $i$ odd, the indices $i$ and $i+1$ should always be considered as connected, or equivalently, as the same node. Moreover, the graph should be undirected, i.e., $A_{\mathrm{adj}}$ should be symmetric. A pseudo-code is given in \cref{alg:Adjacency}.
\begin{algorithm}
	\caption{Adjacency matrix construction}
	\begin{algorithmic}
		\State \textbf{Input:} $A\in\mathbb{R}^{n\times n}$, $n$ even, $\rho>0$.
		\State Define $A_{\mathrm{adj}} = 0_{n\times n}$.
		\For{$i=1:2:n-3$}
			\For{$j=i+2:2:n-1$}
				\State Let $\mathtt{l}_1 = \{i,i+1\}$ and $\mathtt{l}_2=\{j, j+1\}$.
				\If{$\|[A_{\mathtt{l}_1, \mathtt{l}_2}\ | \ A_{\mathtt{l}_2,\mathtt{l}_1}]\|_\mathrm{F}> \sqrt{\rho\|A\|_\mathrm{F}}$}
					\State Set $A_{\mathrm{adj},\mathtt{l}_1\cup \mathtt{l}_2,\mathtt{l}_1\cup \mathtt{l}_2} = \texttt{ones}(4, 4)$.
				\EndIf
			\EndFor
		\EndFor
	    \Return $A_{\mathrm{adj}}$.
	\end{algorithmic}
	\label{alg:Adjacency}
\end{algorithm}

\subsection{Step II.1: Symmetric Skew-Hamiltonian Jacobi} The purpose of step~II.1 is diagonalizing symmetric skew-Hamiltonian submatrices of $A$.

Let $l$ be a list of $2m$ indices obtained from step~I.3 and let $\sigma\geq 0$ be the average singular value of $\Omega_{l,l}$. Let $\widetilde{M}\coloneq A_{l,l}-\sigma (I_m \otimes J_2)$. By \eqref{eq:evp_complex_repeated}, a first possibility is that $\widehat{M}\coloneq P_{\mathrm{eo}}^\top \widetilde{M} P_{\mathrm{eo}}$ is a perturbed symmetric skew-Hamiltonian matrix.
In \cref{thm:nearest_ssh}, we show that letting the partition in $m\times m$ blocks $\widehat{M}=\left[\begin{smallmatrix}\widehat{M}_{11}&\widehat{M}_{12}\\\widehat{M}_{21}&\widehat{M}_{22}\end{smallmatrix}\right]$, the nearest symmetric skew-Hamiltonian matrix is given by
\begin{equation*}
	\argmin_{M\in\mathrm{SSkH}(2m)}\|\widehat{M}-M\|_\mathrm{F} = \mathrm{sskh}(\widehat{M}) \coloneq \frac{1}{2}\begin{bmatrix}
	\mathrm{sym}(\widehat{M}_{11}+\widehat{M}_{22})& -\mathrm{skew}(\widehat{M}_{21}-\widehat{M}_{12})\\
	\mathrm{skew}(\widehat{M}_{21}-\widehat{M}_{12})&\mathrm{sym}(\widehat{M}_{11}+\widehat{M}_{22})
	\end{bmatrix}.
\end{equation*}
Therefore, if it holds that
\begin{equation*}
	\|\widehat{M} - \mathrm{sskh}(\widehat{M})\|_\mathrm{F}\leq \sqrt{\rho \|A\|_\mathrm{F}},
\end{equation*}
then the matrix $\mathrm{sskh}(\widehat{M})$ can be diagonalized using~\cref{alg:ssh_jacobi}, which is an implicit version of the Jacobi-like method introduced in~\cite[Section~5.4]{FABENDER200137}. During~\cref{alg:ssh_jacobi}, the remaining part $\widehat{M} - \mathrm{sskh}(\widehat{M})$ is considered as a noise which is omitted while diagonalizing~$\mathrm{sskh}(\widehat{M})$. The Frobenius norm of the noise is conserved by~\cref{alg:ssh_jacobi}. Hence, it remains small and it is quickly diagonalized at step III using \cref{alg:Zhou}. 

The algorithm presented in~\cite[Section~5.4]{FABENDER200137} applies to $\mathrm{sskh}(\widehat{M})\in\mathrm{SSkH}(2m)$ while~\cref{alg:ssh_jacobi} requires working with the non-permuted form $\widetilde{M}$. Therefore, all operations from~\cite[Section~5.4]{FABENDER200137}, including the submatrices selection and the similarity transformations, have to be carefully permuted. For completeness of this paper, we provide all necessary information to implement~\cref{alg:ssh_jacobi}.
First, we define
\begin{equation}\label{eq:sskh2}
	\mathrm{sskh2}(\widetilde{M})\coloneq P_{\mathrm{eo}}\mathrm{sskh}(P_{\mathrm{eo}}^\top \widetilde{M} P_{\mathrm{eo}})P_{\mathrm{eo}}^\top =P_{\mathrm{eo}}\mathrm{sskh}(\widehat{M})P_{\mathrm{eo}}^\top  .
\end{equation}
Given $i\in\{1,...,2m-3\}$ and $j\in\{i+2,...,2m-1\}$, $i,j$ odd and $\mathtt{l} = \{i,i+1,j,j+1\}$, basic arithmetic operations show that every $4\times4$ submatrix $\mathrm{sskh2}(\widetilde{M})_{\mathtt{l},\mathtt{l}}$ has the form
\begin{equation*}
\mathrm{sskh2}(\widetilde{M})_{\mathtt{l},\mathtt{l}} =\begin{bmatrix}
h_1&0&h_2&\omega\\
0&h_1&-\omega&h_2\\
h_2&-\omega&h_3&0\\
\omega&h_2&0&h_3
\end{bmatrix},\quad  \text{with}\ h_1,h_2,h_3,\omega\in\mathbb{R}.
\end{equation*}
Then, \cite[Section~5.4]{FABENDER200137} shows that defining $p = [-\omega\quad \frac{1}{2}(h_1-h_3)\quad h_2]^\top$, $\alpha = \|p\|_2$ and $\beta = \alpha + p_2$, we have
\begin{equation}\label{eq:ssh_rotation}
 R = \frac{1}{\sqrt{2\alpha\beta}}\begin{bmatrix}
 \beta&0&-p_3&p_1\\
 p_3&p_1&\beta&0\\
 0&\beta&-p_1&-p_3\\
 -p_1&p_3&0&\beta
 \end{bmatrix}\in\mathrm{O}(4).
\end{equation}
The matrix $RP_{\mathrm{eo}}$ is ortho-symplectic and $R$ diagonalizes $\mathrm{sskh2}(\widetilde{M})_{l,l}$ \cite{FABENDER200137}. Moreover,
\begin{align}
\nonumber
	 R^\top (\mathrm{sskh2}(\widetilde{M})_{\mathtt{l},\mathtt{l}} + \sigma I_2\otimes J_2) R& =  R^\top \mathrm{sskh2}(\widetilde{M})_{\mathtt{l},\mathtt{l}} R+ \sigma I_2\otimes J_2 \\
	 \label{eq:diagonalized_ssh2}
	  &= \begin{bmatrix}
 \lambda_1&-\sigma&0&0\\
 \sigma&\lambda_1&0&0\\
 0&0&\lambda_2&-\sigma\\
 0&0&\sigma&\lambda_2
\end{bmatrix}.
\end{align}
Combining \eqref{eq:diagonalized_ssh2} with a sweeping strategy provides a Jacobi-like algorithm for the matrix $\mathrm{sskh2}(\widetilde{M}) +  \sigma (I_m\otimes J_2)$. A pseudo-code is given for the implicit, cyclic procedure in \cref{alg:ssh_jacobi} with $
	\mathrm{offdiag}:\mathbb{R}^{n\times n}\rightarrow \mathbb{R}: A \mapsto \sqrt{\sum_{i=1}^n \sum_{j=i+1}^n a_{i,j}^2 + a_{j, i}^2}$.
\begin{algorithm}
	\caption{Implicit, permuted, cyclic SSH Jacobi-like algorithm}
	\begin{algorithmic}
		\State \textbf{Input:} $A\in\mathbb{R}^{n\times n}$, $Q\in\mathrm{O}(n)$, a list of indices $l\in\mathbb{N}^{2m}_0$, $\rho>0$.
		\While{$\mathrm{offdiag}(\mathrm{sskh2}(A_{l,l})) > \rho \|A\|_\mathrm{F} $}
			\For{$i=1:2:2m-3$}
				\For{$j=i+2:2:2m-1$}
					\State Let $\mathtt{l} = \{l_i,l_{i+1},l_j,l_{j+1}\}$.
					\State Compute $R$ from \eqref{eq:ssh_rotation} for $\mathrm{sskh2}(A_{\mathtt{l},\mathtt{l}})$.
					\State Update $A_{:,\mathtt{l}}\leftarrow A_{:,\mathtt{l}}R$ and $A_{\mathtt{l},:}\leftarrow R^\top A_{\mathtt{l},:}$.
					\State Update $Q_{:,\mathtt{l}}\leftarrow Q_{:,\mathtt{l}}R$.
				\EndFor
			\EndFor
		\EndWhile
		\Return $A$ and $Q$.
	\end{algorithmic}
	\label{alg:ssh_jacobi}
\end{algorithm}
\vspace*{-0.3cm}
\subsection{Step II.2: Implicit Symmetric Jacobi} The purpose of step~II.2 is diagonalizing the symmetric submatrices of~$A$. Let $l$ be a list of $r$ indices obtained from step~I.3 and $\widehat{H} \coloneq A_{l,l}$. Moreover, assume $\|\mathrm{skew}(\widehat{H})\|_\mathrm{F}<\sqrt{\rho\|A\|_\mathrm{F}}$. Then, $\mathrm{skew}(\widehat{H})$ can be considered as a perturbation of the symmetric matrix $\mathrm{sym}(\widehat{H})$. The perturbation error will be diagonalized at step III. At step II.2, $\mathrm{sym}(\widehat{H})$ is diagonalized implicitly using Jacobi's algorithm for symmetric matrices on the rows and columns of~$A$. The pseudo-code is given in~\cref{alg:implicit_symmetric_jacobi}. Again, for completeness, we recall that for $h_{11},h_{12},h_{22}\in\mathbb{R}$, a Jacobi rotation yields
\begin{equation*}
\begin{bmatrix}
	c&s\\
	-s&c
	\end{bmatrix}
	\begin{bmatrix}
	h_{11}&h_{12}\\
	h_{12}&h_{22}
	\end{bmatrix}\begin{bmatrix}
	c&-s\\
	s&c
	\end{bmatrix} =\begin{bmatrix}
	\lambda_1&0\\
	0& \lambda_2
	\end{bmatrix}\quad \text{with}\quad \lambda_1,\lambda_2\in\mathbb{R}.
\end{equation*}
If $h_{12}\neq 0$, the Jacobi rotation is computed by 
\begin{equation*}
	\kappa = \frac{h_{11}-h_{22}}{2h_{12}},\quad t =\frac{\mathrm{sign}(\kappa)}{|\kappa| + \sqrt{1+\kappa^2}},\quad
	c = \frac{1}{\sqrt{1+t^2}}\quad\text{and,}\quad s = ct.
\end{equation*}
\begin{algorithm}[t]
 	\caption{Implicit cyclic symmetric Jacobi's algorithm}
	\begin{algorithmic}
		\State \textbf{Input:} $A\in \mathbb{R}^{n\times n}$, $Q\in\mathrm{O}(n)$, a list $l\in\mathbb{N}^r_0$ and $\rho>0$.
		\While{$\mathrm{offdiag}(\mathrm{sym}(A_{l,l}))>\rho \|A\|_\mathrm{F}$}
		\For{$i=1:r$}
			\For{$j\in i+1:r$}
				\State Let $\mathtt{l}=\{l_i,l_j\}$.
				\State Compute the Jacobi rotation $R$ diagonalizing $H\coloneq \mathrm{sym}(A_{\mathtt{l},\mathtt{l}})$.
				\State Update $A_{:,\mathtt{l}}\leftarrow A_{:,\mathtt{l}} R$ and  $A_{l,:}\leftarrow R^\top A_{\mathtt{l},:}$.
				\State Update $Q_{:,\mathtt{l}}\leftarrow Q_{:,\mathtt{l}} R$.
			\EndFor
		\EndFor
		\EndWhile
		\Return $A$ and $Q$.
	\end{algorithmic}
	\label{alg:implicit_symmetric_jacobi}
\end{algorithm}
\vspace*{-0.3cm}

\subsection{Step III: Accuracy refinement} Steps I and II led to the point where $\mathrm{offschur}(A)$ has the worst-case magnitude of the order of $\sqrt{\rho\|A\|_\mathrm{F}}$, either because small nonzero entries have been neglected by the cluster identification procedure of step I.3 , or because they have been considered as perturbations at step II. Since only orthogonal similarity transformations have been applied to $A$, the matrix remains normal at the end of step~II. One or two sweeps of \cref{alg:Zhou} lead to the block-diagonalization of $A$ with the condition $\mathrm{offschur}(A)\leq\rho\|A\|_\mathrm{F}$. 

\subsection{Convergence of the algorithm} We conclude this section by showing the global convergence of \cref{alg:normalpaardekooper} in \cref{thm:convergence}. Importantly, we assume~\cref{alg:Zhou} converges globally for normal matrices. This assumption is validated by available algorithms for which it is proven \cite[Thm.~9.1]{Goldstine59}, or conjectured \cite{BunseGerstner1993,ZhouBrent2003}.
\begin{theorem}\label{thm:convergence}
Assume~\cref{alg:Zhou} is a Jacobi-like algorithm that always returns for normal matrices in exact arithmetic. Then, \cref{alg:normalpaardekooper} always returns for normal matrices in exact arithmetic and its output $(S,Q)$ satisfies $A = QSQ^\top$ with $\mathrm{offschur}(S)\leq\rho\|A\|_\mathrm{F}$.
\end{theorem}
\begin{proof}
	First, step I.1 always returns. Indeed, the implicit sweeps of \cref{alg:paardekooper} with $A$ as input are identical to Paardekooper's sweeps with $\mathrm{skew}(A)$ as input. Since, Paardekooper's method converges globally for skew-symmetric matrices~\cite{Paardekooper1971}, step I.1 stops after finitely many sweeps.
	
	Steps I.2 and I.3 require $\mathcal{O}(n^2)$ arithmetic operations. They are finite by definition.
	
	If step II.1 is applied, then the sweeps of \cref{alg:ssh_jacobi} with $A_{l_i,l_i}$ as input are equivalent to sweeps of~\cite{FABENDER200137} with $P_{\mathrm{eo}}^\top \mathrm{ssh2}(A_{l_i,l_i})P_{\mathrm{eo}}$ as input.\footnote{Jacobi's algorithm for SSkH matrices is exactly equivalent to Jacobi's algorithm for Hermitian matrices of half the size by the identification \eqref{eq:ssh_hermitian}.} Therefore, it converges globally  and stops after finitely many sweeps.
	
	If step II.2 is applied, then the iterations of \cref{alg:implicit_symmetric_jacobi} are equivalent to iterations of symmetric Jacobi with $\mathrm{sym}(A_{l_i,l_i})$ as input. Therefore, it converges globally and stops after finitely many sweeps.
	
	If step II.3 is applied, then~\cref{alg:Zhou} is not ensured to converge because $A_{l_i,l_i}$ is not normal. By adding a finite time stopping condition, we ensure step II.3 stops in finitely many sweeps.
	
	Finally, at step III, $A$ is normal, thus \cref{alg:Zhou} converges by assumption. 
\end{proof}

\section{Perturbation analysis} \label{sec:floating_arithmetic} This section presents a perturbation analysis yielding analogs of the results of \cref{sec:theorems} that take into account the effect of inexactness. It quantifies how the matrix structures, such as symmetry in~\eqref{eq:evp_real}, are perturbed in practice. This section explains the behavior and justifies the design of~\cref{alg:normalpaardekooper}. More precisely, by means of perturbation bounds, it identifies matrices that may trigger the different steps of the algorithm. Indeed, the results of~\cref{sec:theorems} hold if and only if the real Schur decomposition of $\Omega$ is known exactly. In practice, the inexactness comes from (i) floating-point arithmetic and (ii) the stopping criteria in~\cref{alg:normalpaardekooper}. If the stopping criteria are not smaller than machine precision, we claim that their influence dominates that of floating-point arithmetic, such that the numerical behavior of~\cref{alg:normalpaardekooper} can be anticipated by considering \emph{exact} arithmetic. This claim is supported by the following facts.
\begin{enumerate}
	\item All matrix transformations involve orthogonal matrices, in particular, Givens rotations. These are well-known to be backward stable: given a matrix $\widehat{Q}$ obtained from the floating-point product of $s$ sets of $N$ disjoint Givens rotations and $Q$ the exact product, $\|Q-\widehat{Q}\|_\mathrm{F}\leq \varepsilon_\mathrm{m}f(s)$ where $f(s)$ is a moderately growing function of $s$ \cite[Lem.~19.9]{Higham2002} and $\varepsilon_\mathrm{m}$ is the unit roundoff..
	\item The effect of floating-point arithmetic is captured by the stopping criterion. Indeed, given a stopping criterion $\widehat{\rho} = \varepsilon_\mathrm{m} $ in floating-point arithmetic, there is an effective stopping criterion $\rho$ of equivalent magnitude in exact arithmetic. Indeed, assume the approximate RSD $\widehat{Q}^\top A\widehat{Q} - \widehat{S} =  \widehat{E}$ with $\| \widehat{E}\|_\mathrm{F} \leq   \varepsilon_\mathrm{m}\|A\|_\mathrm{F} $. Let $S$ be an exact Schur form of $A$ that is nearest to $\widehat{S}$ and assume $\widehat{Q}$ is approximately orthogonal such that $ \widehat{Q} = Q+ E_Q$ for $Q\in\mathrm{O}(n)$ with $\|E_Q\|_\mathrm{F}\lesssim \varepsilon_\mathrm{m}$, then
	\begin{align*}
		&\widehat{Q}^\top A\widehat{Q} - \widehat{S} =  \widehat{E}\\
		\iff &(Q + E_Q)^\top A (Q + E_Q) - \widehat{S}= \widehat{E}\\
		\iff& Q^\top A Q - S = E\eqcolon  \widehat{E} - E_Q^\top AQ - Q^\top A E_Q -E_Q^\top A E_Q+ (\widehat{S} - S).
	\end{align*}	 
	This yields an effective error bound $\|E\|_\mathrm{F}\leq \varepsilon_\mathrm{m} \tau\|A\|_\mathrm{F}$ with $\tau\geq 1$ for exact arithmetic. 
\end{enumerate}
Therefore, this section assumes that operations are performed in exact arithmetic with inexactness only induced by the stopping criteria.

Consider the output of step I.1 where the stopping criterion $\rho = \varepsilon_\mathrm{m}$ is used. Then, by~\cite[Lem.~4.3]{mataignegallivan2025}, step I.1 returns $Q\in\mathrm{O}(n)$ such that
\begin{equation}\label{eq:skew_form}
	\Omega Q -Q\mathrm{skew}(S) = E,\ \text{where}\ \| E\|_\mathrm{F}\leq \varepsilon_\mathrm{m}\tau \|A\|_\mathrm{F}\text{ and $S$ is a Schur form of $A$.}
\end{equation}
with $\tau\geq 1$. For the rest of this section, the short-hand notation $\cos(\theta_i)\eqcolon c_i$ and $\sin(\theta_i)\eqcolon s_i$ is used.
\subsection{Eigenvalues with distinct nonzero imaginary parts} Assume first that all eigenvalues of $A$ have distinct nonzero imaginary parts. The eigenvalues are denoted by $\{\lambda_j e^{\pm i\theta_j}\}_{j=1}^p$. Then, if and only if~\eqref{eq:skew_form} holds, it is shown in \cite[Thm.~4.4]{mataignegallivan2025} that
\begin{equation}\label{eq:error}
	\|AQ-QS\|_\mathrm{F} \leq \tau \varepsilon_\mathrm{m} \|A\|_\mathrm{F} \left(1+\max_{\substack{i,j =1,...,p\\ i\neq j}}\left|\frac{\lambda_ic_i - \lambda_jc_j}{\lambda_is_i - \lambda_js_j} \right|\right).
\end{equation}
By \eqref{eq:error}, small distances between the imaginary parts of the eigenvalues of $A$ result in non-negligible off-diagonal blocks in $A$ after step I.1. However, unless these off-diagonal blocks are large enough, they are ignored by the clustering procedure of step~I.2, and thus not considered at step~II. Step III is designed for solving the loss of accuracy induced by \eqref{eq:error}.
A similar result can be obtained for eigenvalues with repeated nonzero imaginary parts and real eigenvalues.
\subsection{Eigenvalues with repeated nonzero imaginary parts}\label{sec:complex_repeated_inexact} Assume there is multiplicity in the imaginary parts of the eigenvalues such that for all $j\in [j]\subseteq \{1,...,p\}$, we have $\lambda_js_j = \sigma$. Let  $m\coloneq |[j]|$ and $[i]\coloneq \{1,...,p\}\setminus [j]$. Then,  by \eqref{eq:skew_form}, there is subset $V\in\mathrm{St}(n,2m)$ of the columns of $Q$ such that $\Omega V - V(J_2\otimes \sigma I_m) = E'$ with $\|E'\|_\mathrm{F}\leq \tau \varepsilon_\mathrm{m}\|A\|_\mathrm{F}$. By \cref{thm:accuracy_repeated_improved} given in appendix, we have
\begin{align}
\label{eq:bound_thm_repeated}
	&\min_{R\in\mathrm{OSp}(2m)}\|A(VR)-(VR)\left[\begin{smallmatrix}
		D_m&-\sigma I_m\\
		\sigma I_m&D_m
		\end{smallmatrix}\right]\|_\mathrm{F} \\
		\nonumber
		&\leq\tau \varepsilon_\mathrm{m} \|A\|_\mathrm{F} \left(1 + \frac{\max\limits_{i\in[i],j\in[j]}|\lambda_i c_i-\lambda_j c_j|}{\min\limits_{i\in[i]}|\lambda_i s_i-\sigma|}+ \frac{\max\limits_{j_1,j_2\in[j]}|\lambda_{j_1} c_{j_1}-\lambda_{j_2} c_{j_2}|}{\sigma}\right)+\mathrm{o}(\varepsilon_\mathrm{m}),
\end{align}
where $D_m$ is an $m\times m $ diagonal matrix containing the real parts of the eigenvalues. Thus, after step I.1, there is a $2m\times 2m$ submatrix $\widehat{M}=V^\top  A V - \sigma J_2\otimes I_m $ and
\begin{align*}
	\min_{M\in\mathrm{SSkH}(2m)}\|\widehat{M}-M\|_\mathrm{F}&\leq\min_{R\in\mathrm{OSp}(2m)}\|\widehat{M} - R\left[\begin{smallmatrix}
		D_m&0\\
		0&D_m
		\end{smallmatrix}\right] R^\top\|_\mathrm{F}\\
		&=\min_{R\in\mathrm{OSp}(2m)}\|\widehat{M} + \sigma J_2\otimes I_m - R\left[\begin{smallmatrix}
		D_m&-\sigma I_m\\
		\sigma I_m&D_m
		\end{smallmatrix}\right] R^\top\|_\mathrm{F}\\
		&\leq \min_{\breve{R}\in\mathrm{OSp}(2m)}\|AV - V (R\left[\begin{smallmatrix}
		D_m&-\sigma I_m\\
		\sigma I_m&D_m
		\end{smallmatrix}\right] R^\top)\|_\mathrm{F}.
\end{align*}
Combining the latter inequality with \eqref{eq:bound_thm_repeated}, we obtain
\begin{align}
\label{eq:dist_to_ssh}
&\|\widehat{M}-\mathrm{sskh}(\widehat{M})\|_\mathrm{F}\\
\nonumber
&\leq\varepsilon_\mathrm{m} \tau \|A\|_\mathrm{F}\left(1 + \frac{\max\limits_{i\in\mathcal{I},j\in[j]}|\lambda_i c_i-\lambda_j c_j|}{\min\limits_{i\in\mathcal{I}}|\lambda_i s_i-\sigma|}+\frac{\max\limits_{j_1,j_2\in[j]}|\lambda_{j_1} c_{j_1}-\lambda_{j_2} c_{j_2}|}{\sigma}\right)+ \mathrm{o}(\varepsilon_\mathrm{m}).
\end{align}
In conclusion, $\widehat{M} = \mathrm{sskh}(\widehat{M}) + \Delta\widehat{M}$ where the error bound on $\Delta\widehat{M}$ is amplified, either by eigenvalues with imaginary parts close to $\sigma$ or if $\sigma$ is very small. The consequence for \cref{alg:normalpaardekooper} is that the specialized Jacobi algorithm for SSkH matrices of step~II.1 is only triggered if $\Delta\widehat{M}$ is small enough. Otherwise, the method is redirected to step~II.3 and then to step~III.

\subsection{Real eigenvalues} Finally, assume the matrix $A$ has $r$ real eigenvalues and define the sets $[j]\coloneq\{1,...,p\}$ and $[r]\coloneq\{1,...,r\}$. Then, by \eqref{eq:skew_form}, there is a subset $V\in\mathrm{St}(n,r)$ of columns of $Q$ such that $\Omega V=E''$ with $\|E''\|_\mathrm{F}\leq \tau \varepsilon_\mathrm{m} \|A\|_\mathrm{F}$. By \cite[Thm.~4.7]{mataignegallivan2025}, it holds that
\begin{equation}\label{eq:real_perturbation}
	\min_{\breve{R}\in\mathrm{O}(r)}\|A(V\breve{R}) - (V \breve{R})\breve{\Lambda} \|_\mathrm{F}\leq \tau \varepsilon_\mathrm{m} \|A\|_\mathrm{F} \left(1 + \frac{\max\limits_{j\in[j], \ k\in[r]}|\lambda_j c_j -\breve{\lambda}_k|}{\min\limits_{j=1,....,p}|\lambda_j s_j|}\right) + \mathrm{o}(\varepsilon_\mathrm{m}).
\end{equation}
Similarly, the error bound on the real eigenspace is amplified by complex eigenvalues with small imaginary parts. The result from \eqref{eq:real_perturbation} is the perturbed analog of \eqref{eq:evp_real}. 

At the end of step I.1, an $r\times r$ submatrix of $A$ is equal to  $\widehat{H}=V^\top  A V$. This submatrix $\widehat{H}$ satisfies
\begin{equation*}
	\min_{H\in\mathrm{Sym}(r)}\|\widehat{H}-H\|_\mathrm{F}\leq\min_{\breve{R}\in\mathrm{O}(r)}\|\widehat{H} - \breve{R}\breve{\Lambda} \breve{R}^\top\|_\mathrm{F}\leq \min_{\breve{R}\in\mathrm{O}(r)}\|AV - V (\breve{R}\breve{\Lambda} \breve{R}^\top)\|_\mathrm{F}.
\end{equation*}
Moreover, it is well known that $ \argmin_{H\in\mathrm{Sym}(r)}\|\widehat{H}-H\|_\mathrm{F}=\mathrm{sym}(\widehat{H})$. Therefore, by~\eqref{eq:real_perturbation}, it follows that 
\begin{equation}\label{eq:dist_to_symmetric}
	\|\widehat{H} - \mathrm{sym}(\widehat{H})\|_\mathrm{F} \leq \tau  \varepsilon_\mathrm{m} \|A\|_\mathrm{F} \left(1 + \frac{\max\limits_{j\in[j], k\in[r]}|\lambda_j c_j -\breve{\lambda}_k|}{\min\limits_{j\in [j]}|\lambda_j s_j|}\right) + \mathrm{o}(\varepsilon_\mathrm{m}).
\end{equation}
In conclusion, $\widehat{H} = \mathrm{sym}(\widehat{H}) + \Delta \widehat{H}$ where the error bound on $\Delta \widehat{H}=\mathrm{skew}(\widehat{H})$ is amplified by eigenvalues with small imaginary parts. If these imaginary parts are large enough, then $\Delta\widehat{H}$ is small enough for \cref{alg:normalpaardekooper} to trigger step II.2. Otherwise, the real eigenvalues are approximated at step II.3 and refined at step III.


\section{Numerical experiments}\label{sec:numerical_experiments}
This section presents numerical experiments regarding the accuracy and the running time of \cref{alg:normalpaardekooper} for various distributions of normal matrices. The codes to reproduce all experiments are available at the address \url{https://github.com/smataigne/NormalJacobi.jl}. We compare the performance of \cref{alg:normalpaardekooper} with other Jacobi-like algorithms for solving the normal eigenvalue problem. The considered methods are:
\begin{enumerate}
	\item Our method: \cref{alg:normalpaardekooper}.
	\item Algorithm~2 from Bunse-Gerstner et al.~\cite{BunseGerstner1993}.\footnote{\cite[Alg.~2]{BunseGerstner1993} does not provide an explicit and unique implementation of the algorithm. Therefore, we choose to implement the one suggested by~\cite[Thm.~4.2]{BunseGerstner1993}.}
	\item The norm-reducing Jacobi procedure of Goldstine and Horwitz~\cite{Goldstine59}.
	\item The Jacobi-like algorithm of Zhou and Brent~\cite{ZhouBrent2003} where the subroutine for the $4\times 4$ real Schur decomposition is the \texttt{lapack} function \href{https://www.netlib.org/lapack/explore-html/d5/d38/group__gees_gab48df0b5c60d7961190d868087f485bc.html}{\texttt{gees}}.
	\item A Jacobi-like implementation of the method \texttt{RandDiag} from He and Kressner~\cite{He2025}. \texttt{RandDiag} is described for a more general context than Jacobi-like algorithms, but it can easily be adapted for this framework, as done in~\cref{sec:randdiag}.
\end{enumerate}
\subsection{Numerical accuracy} Let us first compare the numerical accuracy. The algorithms are stopped either if the stopping criterion $\mathrm{offschur}(A) < 10\varepsilon_\mathrm{m}\|A\|_\mathrm{F}$ ($\rho = 10 \varepsilon_\mathrm{m}$) is satisfied or if a sweep doesn't decrease  $\mathrm{offschur}(A)$.

In~\cref{tab:accuracy}, the numerical results for five distributions of normal matrices are reported. Random normal matrices are generated from the real Schur decomposition
\begin{equation*}
	A  = Q\begin{bmatrix}
	\Lambda\cos(\Theta) \otimes I_2+ \Lambda\sin(\Theta)\otimes J_2&0\\
	0&\breve{\Lambda}
	\end{bmatrix}Q^\top,
\end{equation*} 
by specifying the distribution of the Schur vectors $Q$ and that of the eigenvalues. Let $\mathcal{N}(\mu,\sigma^2)$ denote the normal distribution centered in $\mu\in\mathbb{R}$ with standard deviation $\sigma>0$ and $\mathcal{U}(a,b)$ denote the uniform distribution between $a(\leq b)$ and $b$. The five distributions of test matrices for the experiments are:
\begin{itemize}
	\item[Exp1.] Haar-distributed orthogonal matrices~\cite{Anderson87,Stewart80}, i.e., ``uniformly'' distributed in $\mathrm{O}(n)$ with eigenvalue phases uniformly distributed in $(-\pi,\pi]$.
	\item[Exp2.] Random normal matrices with complex eigenvalues. The Schur vectors are Haar-distributed. The eigenvalues have phases sampled from $\mathcal{U}(0,2\pi)$ and radii sampled from  $\mathcal{U}(0,2)$.
	\item[Exp3.] Same as Exp2 with $30\%$ real eigenvalues. The real eigenvalues are sampled from a standard normal distribution $\mathcal{N}(0,1)$.
	\item[Exp4.] Same as Exp2 with $30\%$ eigenvalues with repeated nonzero imaginary parts. The real parts of the repeated eigenvalues are sampled from a standard normal distribution $\mathcal{N}(0,1)$. The repeated imaginary part is also sampled from this distribution.
	\item[Exp5.] Same as Exp2 with phases sampled from $\pi\sqrt{\varepsilon_\mathrm{m}}\mathcal{N}(1,1)$.
\end{itemize}
We can observe from the results of \cref{tab:accuracy} that \cref{alg:normalpaardekooper} always achieves the target accuracy, similarly to all other algorithms but one (\texttt{RandDiag}). Indeed, as shown in~\cite[Cor.~3]{He2025}, the accuracy of \texttt{RandDiag} deteriorates probabilistically as the matrix-size grows.
\begin{table}[ht]
	\centering
    \begin{tabular}{||c||c|c|c|c|c||}
    \hline
    $n$&\multicolumn{5}{c||}{$\mathrm{offschur}(A)\ /\ \|A\|_\mathrm{F}$}\\
    \hline
	& \cref{alg:normalpaardekooper}& B.-G.~et~al. &Goldstine et al.& Zhou et al.& He et al. \\
	\hline
	\multicolumn{6}{||l||}{Exp1. Haar-distributed orthogonal matrices.}\\
	\hline
	64&1.2e-15&3.6e-16&1.4e-15&1.8e-16&1.4e-12\\ 
128&1.6e-15&5.5e-16&2.2e-15&2.3e-16&3.9e-12\\
256&2.1e-15&8.7e-16&3.3e-15&3.9e-16&6.3e-12\\
512&3.0e-15&1.1e-15&5.0e-15&4.3e-16&3.1e-11\\
	\hline
	\multicolumn{6}{||l||}{Exp2. Random normal matrices with complex eigenvalues.}\\
	\hline
64&1.4e-15&5.2e-16&1.0e-15&4.8e-16&1.9e-13\\
128&2.3e-15&3.8e-16&1.5e-15&4.2e-16&7.2e-13\\
256&3.1e-15&5.7e-16&2.2e-15&4.7e-16&5.9e-12\\
512&4.5e-15&9.3e-16&3.3e-15&7.6e-16&1.3e-11\\
\hline
\multicolumn{6}{||l||}{Exp3. Same as Exp2 with $30\%$ real eigenvalues.}\\
\hline
64&1.6e-15&3.3e-16&1.1e-15&4.0e-16&1.7e-13\\
128&2.2e-15&4.8e-16&1.4e-15&1.9e-13&7.4e-13\\
256&3.7e-15&7.7e-16&2.0e-15&7.5e-16&4.6e-12\\
512&5.1e-15&1.2e-15&2.9e-15&1.2e-15&1.2e-11\\
\hline
\multicolumn{6}{||l||}{Exp4. Same as Exp2 with $30\%$ eig.~with repeated nonzero imaginary parts.}\\
\hline
64&1.5e-15&2.2e-16&9.5e-16&2.8e-16&2.9e-13\\
128&2.6e-15&3.6e-16&1.4e-15&3.8e-16&7.9e-13\\
256&3.4e-15&5.3e-16&2.3e-15&4.3e-16&3.7e-12\\
512&4.7e-15&8.7e-16&3.5e-15&7.3e-16&2.6e-11\\
\hline
\multicolumn{6}{||l||}{Exp5. Same as Exp2 with phases sampled from $\pi\sqrt{\varepsilon_\mathrm{m}}\mathcal{N}(1,1)$.}\\
\hline
64&5.8e-16&3.5e-16&1.5e-15&5.3e-16&7.9e-15\\
128&7.8e-16&5.4e-16&2.1e-15&5.9e-16&8.2e-15\\
256&1.0e-15&1.0e-15&2.6e-15&6.6e-16&9.0e-15\\
512&1.3e-15&1.2e-15&3.5e-15&8.6e-16&1.1e-14\\
	\hline
    \end{tabular}
	\caption{Numerical experiments on the accuracy in double precision of \cref{alg:normalpaardekooper} w.r.t.~the other Jacobi-like algorithms. For distributions of random matrices $A$ specified in \cref{sec:numerical_experiments}, the table shows the output result $\mathrm{offschur}(A)/\|A\|_\mathrm{F}$. We vary the size of the test matrices for $n\in\{64,128,256,512\}$. Every reported accuracy is computed by a geometric mean on $10$ runs.}
	\label{tab:accuracy}
	
	\vspace*{-0.6cm}
\end{table}
\subsection{Running time experiments}
This section compares the running time performance of the five Jacobi-like algorithms listed previously. The algorithms are tested on normal matrices generated from distributions influencing the computational cost of step II in~\cref{alg:normalpaardekooper}. The parameters are:
\begin{itemize}
\item The proportion $\alpha_1\in[0,1]$ of real eigenvalues.  This parameter influences the computational cost of step II.2.
\item The proportion $\alpha_2\in[0,1-\alpha_1]$ of eigenvalues with repeated nonzero imaginary parts.  This parameter influences the computational cost of step II.1 and step~II.3.
\end{itemize}
The Schur vectors $Q$ are Haar-distributed in $\mathrm{O}(n)$. The real eigenvalues are sampled from a standard normal distribution $ \mathcal{N}(0,1)$ as well as the real and the imaginary parts of the complex eigenvalues. The repeated imaginary part of the complex eigenvalues is sampled from $ \mathcal{N}(0,1)$ as well. The running time performance of the four algorithms is displayed in \cref{fig:vary_alpha_12} as the dimension $n$ varies for different pairs $(\alpha_1,\alpha_2)$. For all test matrices, \cref{alg:normalpaardekooper} is the fastest method by a factor $5$ to $10$. This is due to the fact that most arithmetic operations of \cref{alg:normalpaardekooper} rely on  fast-converging methods such as Paardekooper's method. Slower methods such as \cref{alg:Zhou} are only used for final accuracy-improvement computations.

\begin{figure}[ht]
	\centering
	\includegraphics[width = 0.48 \textwidth]{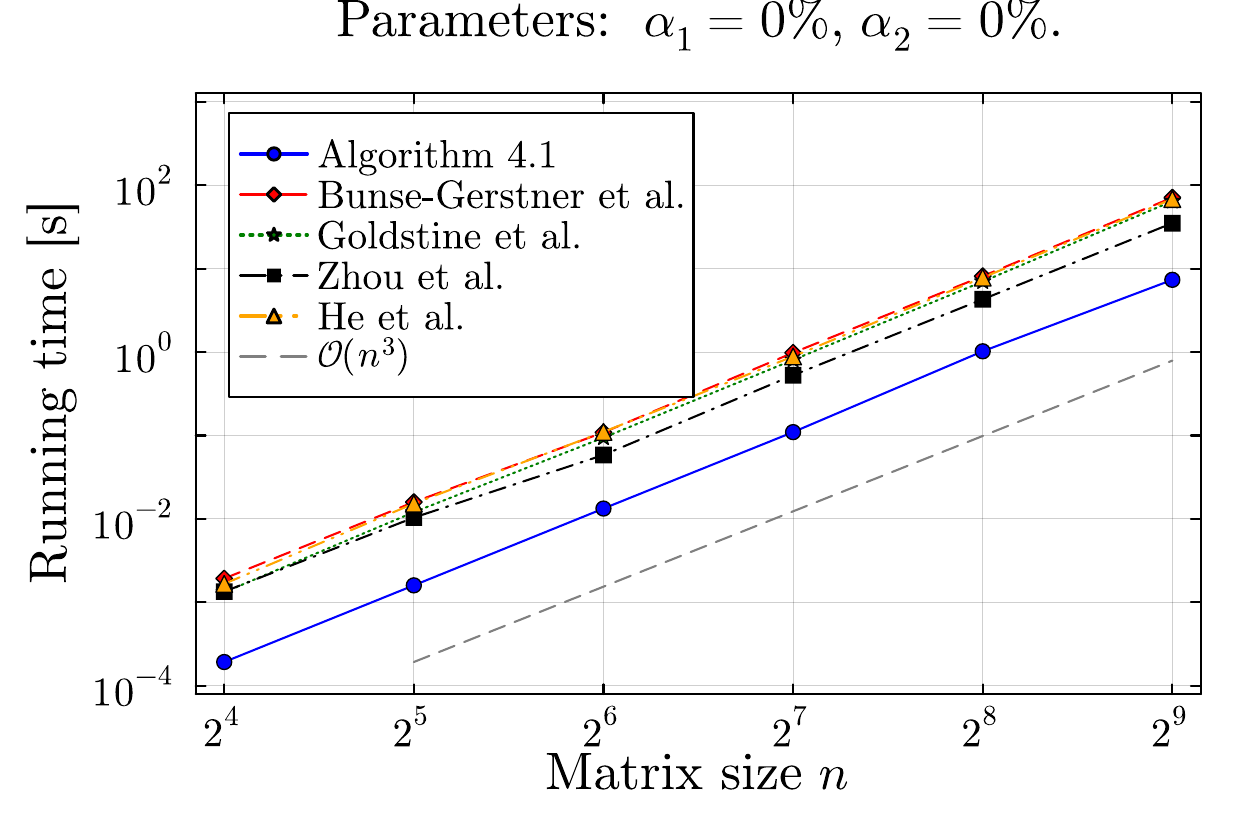}
	\includegraphics[width = 0.48 \textwidth]{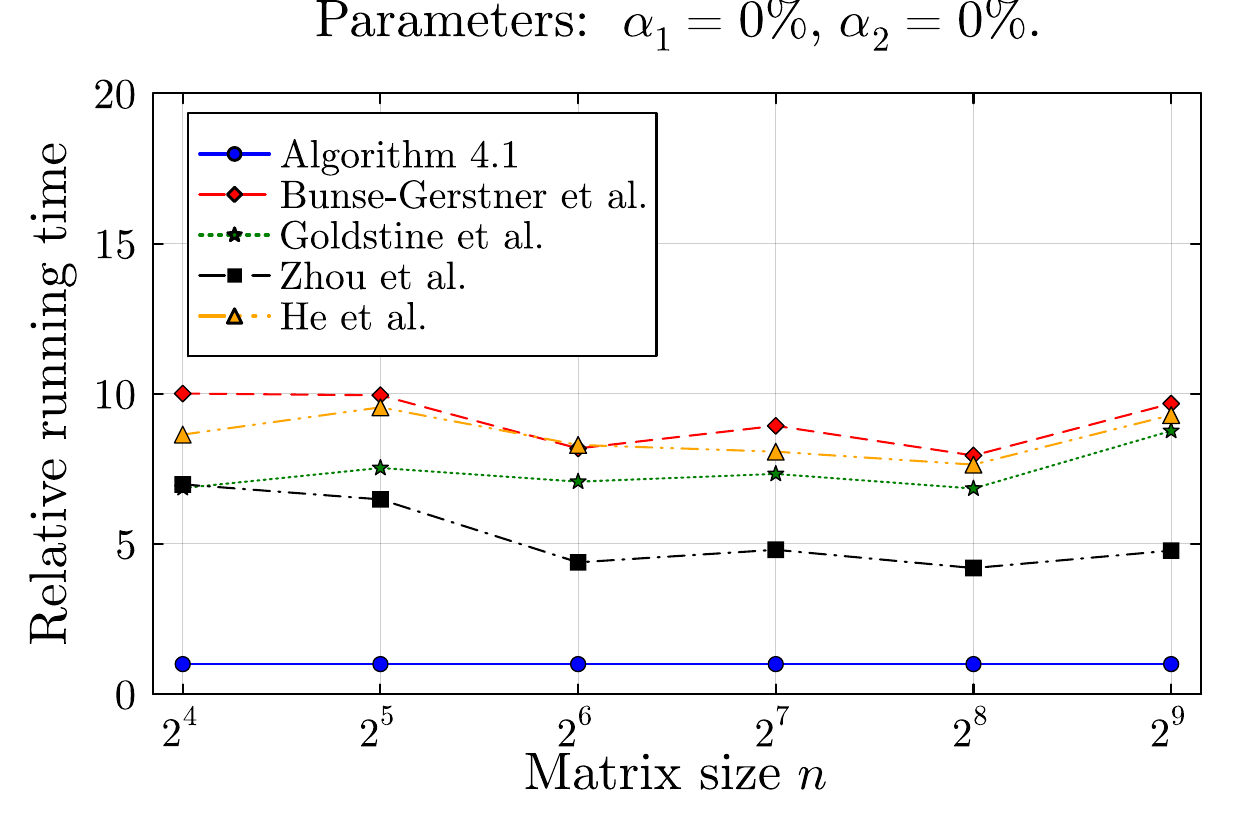}
	\includegraphics[width = 0.48 \textwidth]{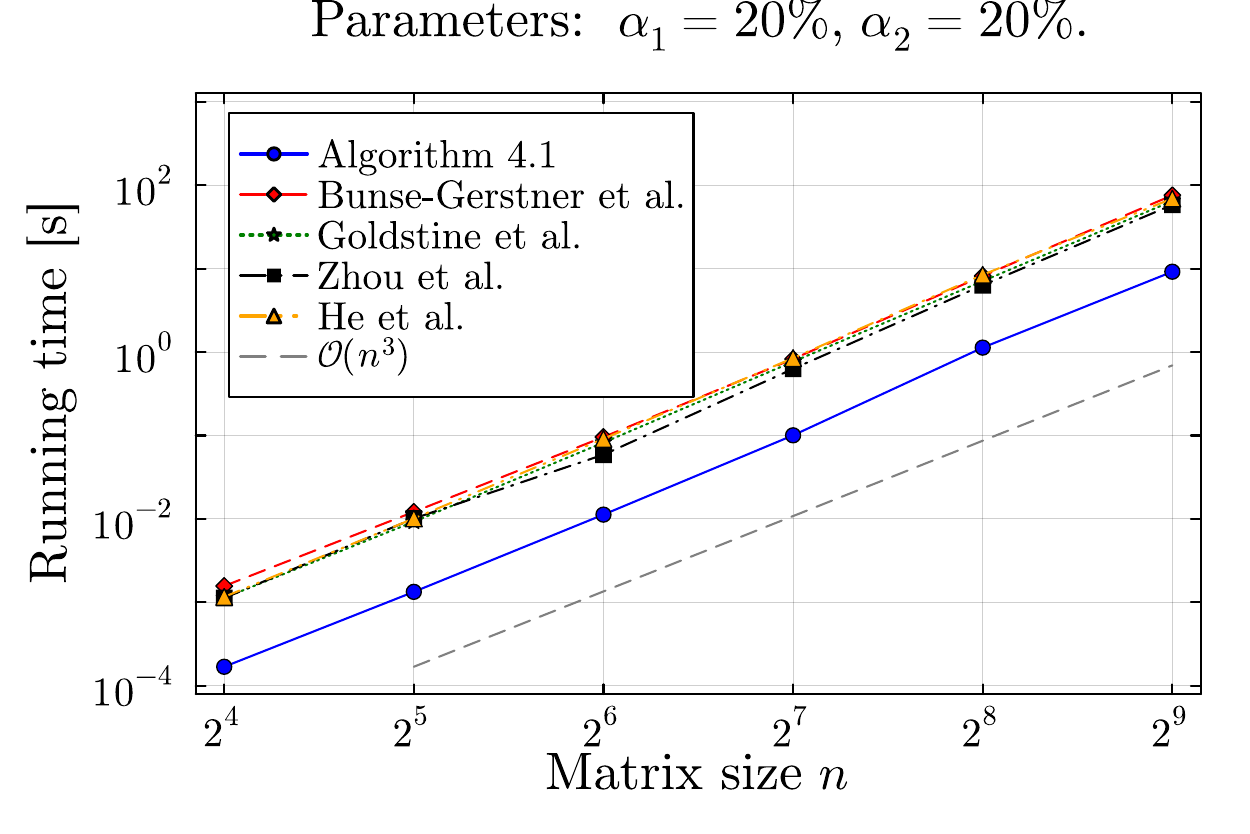}
	\includegraphics[width = 0.48 \textwidth]{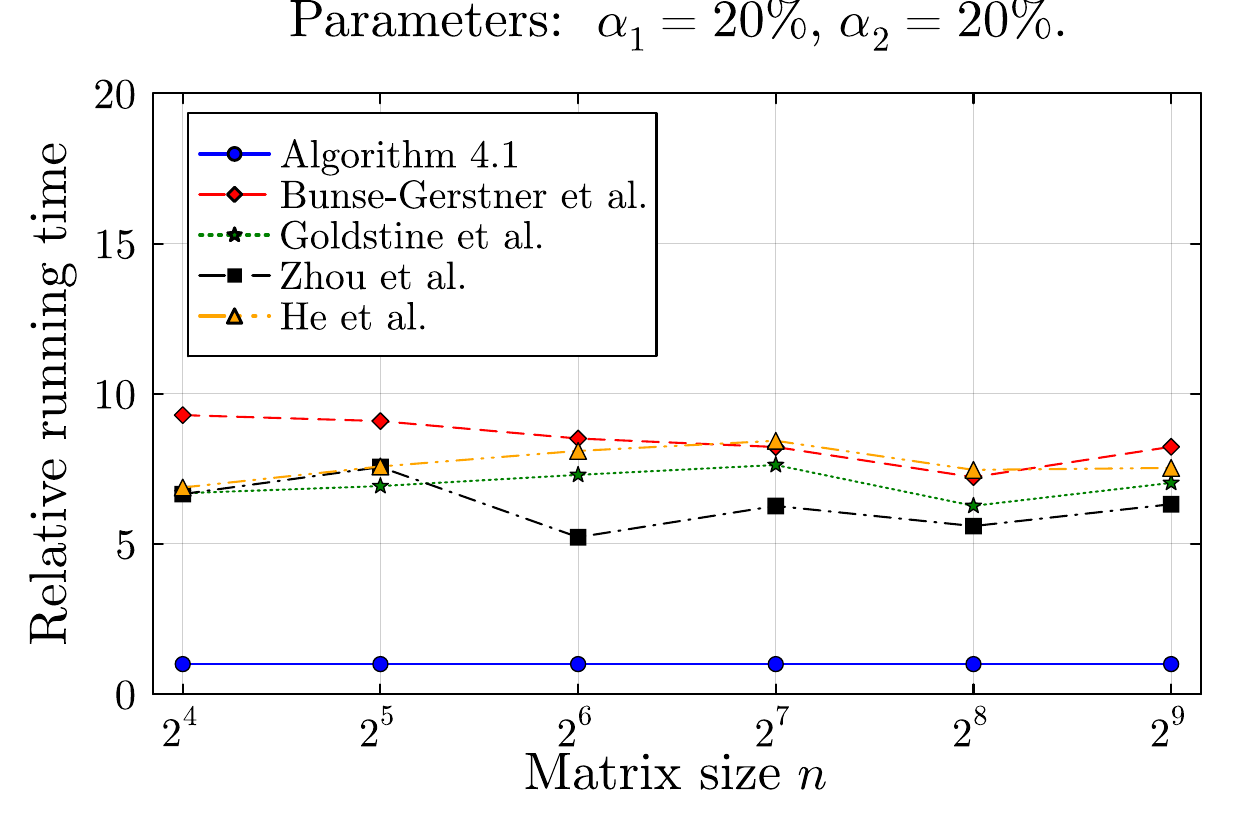}
	\includegraphics[width = 0.48 \textwidth]{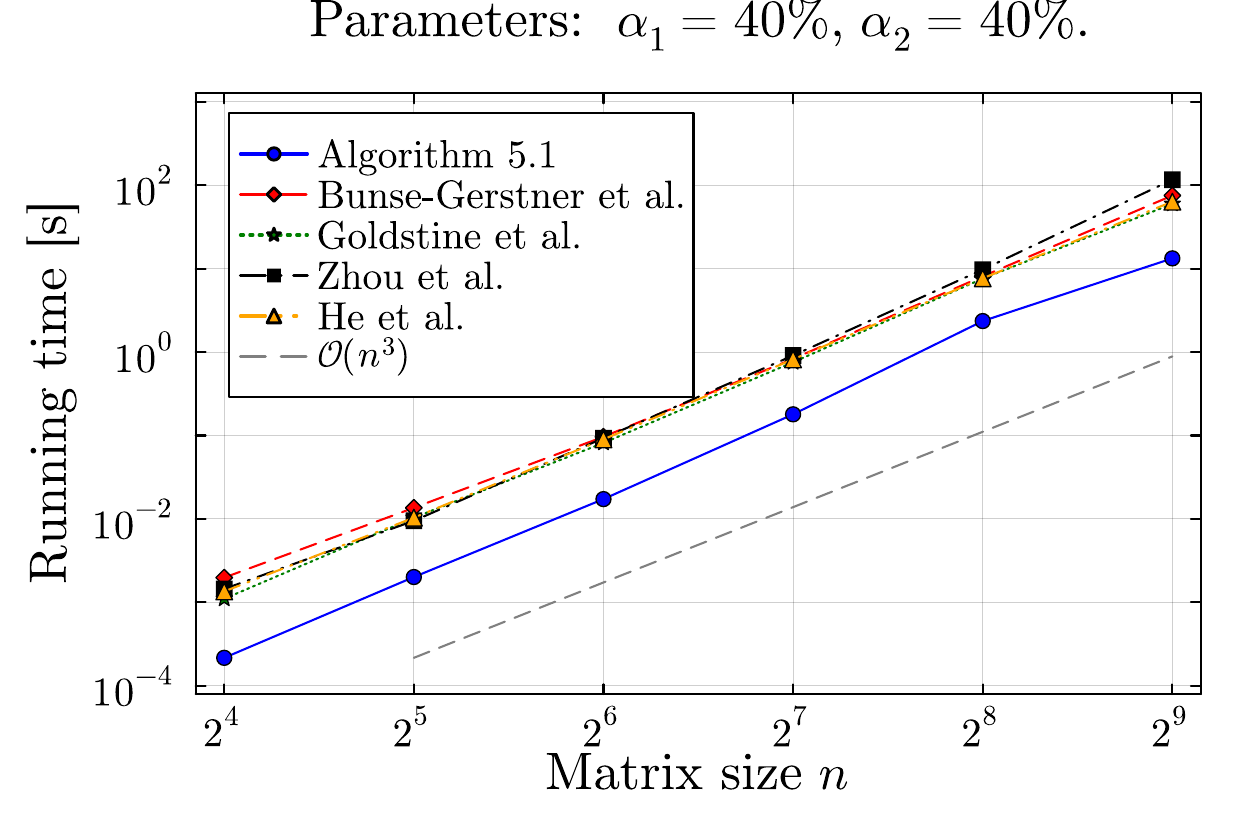}
	\includegraphics[width = 0.48 \textwidth]{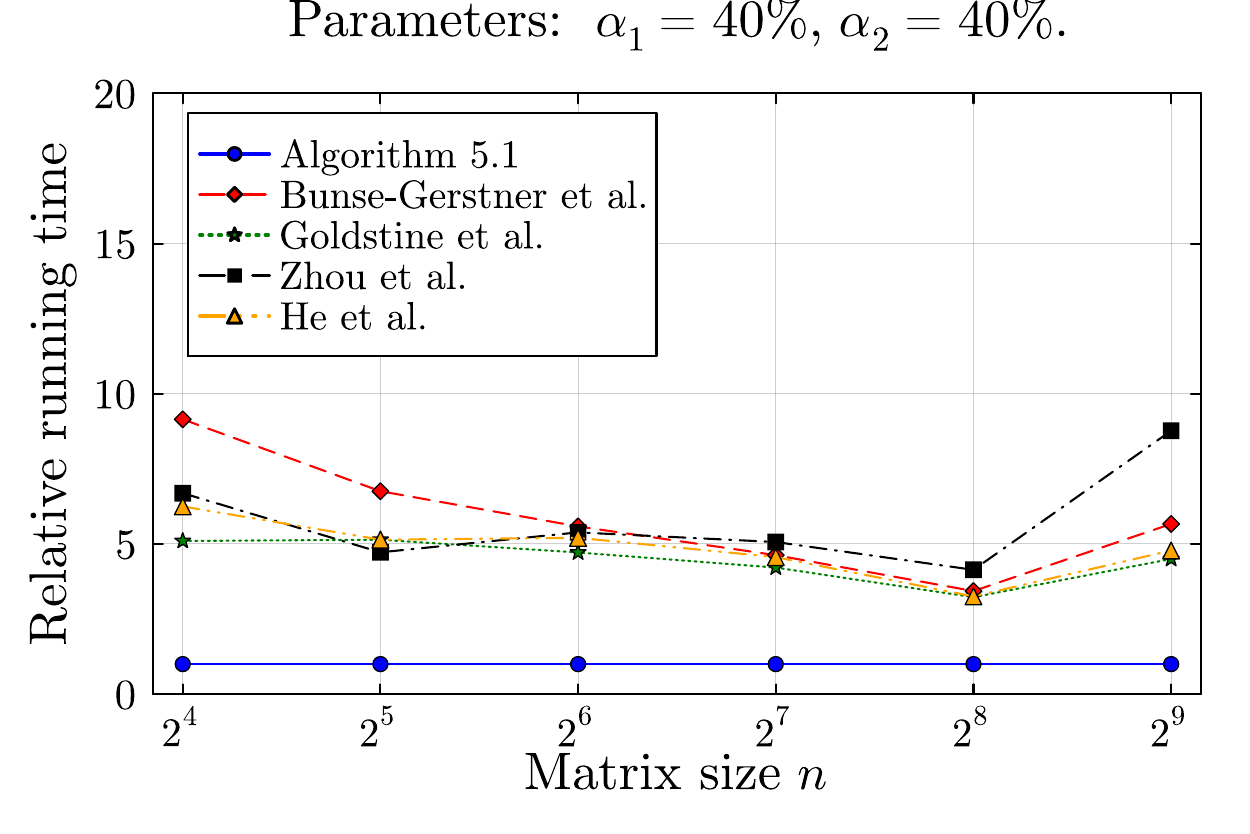}
	\vspace*{-0.3cm}
	\caption{Comparison of the running time of \cref{alg:normalpaardekooper} w.r.t.~the other Jacobi-like algorithms as the matrix size $n$ varies. The parameters $\alpha_1$ and $\alpha_2$ represent, respectively, the proportion of real eigenvalues and that of repeated nonzero imaginary parts of the eigenvalues. On the left, the curves show that absolute running time (in seconds). On the right, the curves show the relative running time w.r.t.~\cref{alg:normalpaardekooper}.}
	\label{fig:vary_alpha_12}
	\vspace*{-0.3cm}
\end{figure}
\section{Nearest symmetric skew-Hamiltonian and ortho-symplectic matrices}\label{sec:SSH_and_os} The design of \cref{alg:normalpaardekooper} and the perturbation analysis of~\cref{sec:floating_arithmetic} required fundamental results on symmetric skew-Hamiltoninan and ortho-symplectic matrices. This section proves these results. Given a matrix $A\in\mathbb{R}^{2m\times 2m}$, let us consider the following partition in $m\times m$ blocks:
\begin{equation*}
	A = \begin{bmatrix}
	A_{11}&A_{12}\\
	A_{21}&A_{22}
	\end{bmatrix}.
\end{equation*} 
\subsection{Nearest symmetric skew-Hamiltonian matrix}
We consider the set $\mathrm{SSkH}(2m)$ of $2m\times 2m$ symmetric skew-Hamiltonian matrices. It can be identified with $m\times m$ Hermitian matrices. Indeed,
\begin{equation}\label{eq:ssh_hermitian}
	\begin{bmatrix}
	\widetilde{H}&-\widetilde{\Omega}\\
	\widetilde{\Omega}&\widetilde{H}
	\end{bmatrix} \in \mathrm{SSkH}(2m)\iff \widetilde{H} + i \widetilde{\Omega} = (\widetilde{H} + i \widetilde{\Omega})^*.
\end{equation}
However, since $\mathbb{R}^{2m\times 2m}$ can not be identified with $\mathbb{C}^{m\times m}$, the nearest symmetric skew-Hamiltonian matrix is not equivalent to the nearest Hermitian matrix.
\begin{theorem}\label{thm:nearest_ssh}
For every matrix $A\in \mathbb{R}^{2m\times 2m}$, the nearest symmetric skew-Hamiltonian matrix to $A$ exists, is unique and is given by
\begin{equation*}
	\argmin_{M\in\mathrm{SSkH}(2m)}\|A - M\|_\mathrm{F} = \mathrm{sskh}(A) \coloneq \frac{1}{2}\begin{bmatrix}
	\mathrm{sym}(A_{11}+A_{22})& -\mathrm{skew}(A_{21}-A_{12})\\
	\mathrm{skew}(A_{21}-A_{12})&\mathrm{sym}(A_{11}+A_{22})
	\end{bmatrix}.
\end{equation*}
\end{theorem}
\begin{proof}
Proving the statement first requires splitting the matrix $A$ as follows:
\begin{align*}
	2A &= (A + J_{2m}AJ_{2m}^\top) + (A - J_{2m}AJ_{2m}^\top)\\
	&=\begin{bmatrix}
	A_{11}+A_{22}&A_{12}-A_{21}\\
	A_{21}-A_{12}&A_{11}+A_{22}
	\end{bmatrix} + \begin{bmatrix}
	A_{11}-A_{22}&A_{12}+A_{21}\\
	A_{21}+A_{12}&A_{22}-A_{11}
	\end{bmatrix}.
\end{align*}
Moreover, a second splitting in symmetric and skew-symmetric parts yields
{\small
\begin{align*}
	\begin{bmatrix}
	A_{11}+A_{22}&A_{12}-A_{21}\\
	A_{21}-A_{12}&A_{11}+A_{22}
	\end{bmatrix} &=\\
	\underbrace{\begin{bmatrix}
	\mathrm{sym}(A_{11}+A_{22})& -\mathrm{skew}(A_{21}-A_{12})\\
	\mathrm{skew}(A_{21}-A_{12})&\mathrm{sym}(A_{11}+A_{22})
	\end{bmatrix}}_{\eqcolon 2\mathrm{sskh}(A)}&+\begin{bmatrix}
	\mathrm{skew}(A_{11}+A_{22})& -\mathrm{sym}(A_{21}-A_{12})\\
	\mathrm{sym}(A_{21}-A_{12})&\mathrm{skew}(A_{11}+A_{22})
	\end{bmatrix}.
\end{align*}}The matrix $\mathrm{sskh}(A)$ is termed the \emph{symmetric skew-Hamiltonian part} of $A$. Since $\mathrm{SSkH}(2m)$ is a linear subspace of $\mathbb{R}^{2m\times 2m}$, $\mathrm{sskh}(A)$ is the nearest SSH matrix to $A$ if and only if $E\coloneq A - \mathrm{sskh}(A)$ is orthogonal to $\mathrm{SSkH}(2m)$. Notice that
{\small
\begin{align*}
	2E&= \underbrace{\begin{bmatrix}
	\mathrm{skew}(A_{11}+A_{22})& -\mathrm{sym}(A_{21}-A_{12})\\
	\mathrm{sym}(A_{21}-A_{12})&\mathrm{skew}(A_{11}+A_{22})
	\end{bmatrix}}_{\eqcolon 2E_1}+\underbrace{\begin{bmatrix}
	A_{11}-A_{22}&A_{12}+A_{21}\\
	A_{21}+A_{12}&A_{22}-A_{11}
	\end{bmatrix}}_{\eqcolon 2E_2}.
\end{align*}}For all $M\in\mathrm{SSkH}(2m)$, by linearity of the inner product, we have
\begin{align*}
	\langle M, 2E\rangle_\mathrm{F} &= \langle M, 2E_1\rangle_\mathrm{F} + \langle M, 2E_2\rangle_\mathrm{F}.
\end{align*}
Let $M= I_2\otimes \widetilde{H} +J_2 \otimes\widetilde{\Omega}$ with $\widetilde{H}\in\mathrm{Sym}(m)$ and $\widetilde{\Omega}\in\mathrm{Skew}(m)$. This yields
\begin{align*}
	\langle M, 2E_1\rangle_\mathrm{F} &= 2\mathrm{Tr}\left(\widetilde{H}\mathrm{skew}(A_{11}+A_{22})-\widetilde{\Omega}\mathrm{sym}(A_{21}-A_{12})\right) =2(0 -0)=0.
\end{align*}
Similarly, we have
\begin{align*}
	\langle M, 2E_2\rangle_\mathrm{F}&=\mathrm{Tr}(\widetilde{H}(A_{11} - A_{22})- \widetilde{\Omega} (A_{21}+A_{12})+\widetilde{H}(A_{22} - A_{11})+ \widetilde{\Omega} (A_{21}+A_{12}))\\
	&=\mathrm{Tr}(\widetilde{H}0 - \widetilde{\Omega} 0)\\
	&=0.
\end{align*}
Therefore, for all $M\in\mathrm{SSkH}(2m)$, $\langle\underbrace{\mathrm{sskh}(A) - M}_{\in\mathrm{SSkH}(2m)}, E\rangle_\mathrm{F}=0$. Thus, it holds that 
\begin{align*}
	\argmin_{M\in\mathrm{SSkH}(2m)}\|A - M\|_\mathrm{F}^2 &=  \argmin_{M\in\mathrm{SSkH}(2m)}\|\mathrm{sskh}(A) - M\|_\mathrm{F}^2+ \|E\|_\mathrm{F}^2\\
	&= \mathrm{sskh}(A).
\end{align*}
\end{proof}
\subsection{Nearest ortho-symplectic matrix}
Consider the set $\mathrm{OSp}(2m)$ of $2m\times 2m$ ortho-symplectic matrices. It is known that $\mathrm{OSp}(2m)$ can be identified with $\mathrm{U}(m)$, the set of $m\times m$ unitary matrices:
\begin{equation*}
\begin{bmatrix}
	U_\mathrm{r}&-U_\mathrm{i}\\
	U_\mathrm{i}&U_\mathrm{r}
	\end{bmatrix} \in\mathrm{OSp}(2m) \iff U_\mathrm{r}+iU_\mathrm{i} \in \mathrm{U}(m).
\end{equation*}
Again, the nearest ortho-symplectic matrix does not readily derive from this identification because the set $\mathbb{R}^{2m\times 2m}$ is larger than $\mathbb{C}^{m\times m}$. In \cref{thm:nearest_os}, an explicit expression for the nearest ortho-symplectic matrix is given. 
From the result of \cref{thm:nearest_os}, an easily computable upper bound on the distance to the ortho-symplectic group is derived in \cref{thm:bound_on_os}. In the absence of ambiguity, we let $I\coloneq I_{2m}$ and $J\coloneq J_2 \otimes I_m$.
\begin{theorem}\label{thm:nearest_os}
	For every matrix $A\in \mathbb{R}^{2m\times 2m}$, $R^\star$ is a nearest ortho-symplectic matrix, i.e., such that
\begin{equation*}
	R^\star =  \begin{bmatrix}
	U_\mathrm{r}^\star&-U_\mathrm{i}^\star\\
	U_\mathrm{i}^\star&U_\mathrm{r}^\star
	\end{bmatrix}\in \argmin_{R\in\mathrm{OSp}(2m)}\|A-R\|_\mathrm{F},
\end{equation*}
if and only if $U_\mathrm{r}^\star+iU_\mathrm{i}^\star = UV^*$ is a polar factor of $\left(\frac{A_{11} + A_{22}}{2}\right) + i\left(\frac{A_{21} - A_{12}}{2}\right)=U\Sigma V^*$. Moreover, 
\begin{equation}\label{eq:os_distance}
	\min_{R\in\mathrm{OSp}(2m)}\|A-R\|_\mathrm{F}^2 = 2\|\Sigma - I\|_\mathrm{F}^2 + \frac14 \left\|AJ_{} - J_{}A\right\|_\mathrm{F}^2.
\end{equation}
\end{theorem}
\begin{proof}
The result follows from the following sequence of identities, where, in particular, the first identity follows from $\langle A + JAJ^\top -R,A-JAJ^\top\rangle_\mathrm{F}=0$.
{\small
\begin{align*}
	\min_{R\in\mathrm{OSp}(2m)}\|A-R\|_\mathrm{F}^2&= \small\min_{R\in\mathrm{OSp}(2m)}\left\|\frac{A + J_{}AJ_{}^\top}{2}-R\right\|_\mathrm{F}^2+\left\|\frac{A - J_{}AJ_{}^\top}{2}\right\|_\mathrm{F}^2\\
	&= \min_{U_\mathrm{r}+iU_\mathrm{i}\in\mathrm{U}(m)}\left\|\frac12 \small\begin{bmatrix}
	A_{11}+A_{22}&A_{12}-A_{21}\\
	A_{21}-A_{12}&A_{11}+A_{22}
	\end{bmatrix}-\small\begin{bmatrix}
	U_\mathrm{r}&-U_\mathrm{i}\\
	U_\mathrm{i}&U_\mathrm{r}
	\end{bmatrix}\right\|_\mathrm{F}^2\\&+\frac14 \left\|AJ_{} - J_{}A\right\|_\mathrm{F}^2\\
	&= \min_{U_\mathrm{r}+iU_\mathrm{i}\in\mathrm{U}(m)}2\left\|\frac12 \begin{bmatrix}
	A_{11}+A_{22}\\
	A_{21}-A_{12}
	\end{bmatrix}-\begin{bmatrix}
	U_\mathrm{r}\\
	U_\mathrm{i}
	\end{bmatrix}\right\|_\mathrm{F}^2+\frac14 \left\|AJ_{} - J_{}A\right\|_\mathrm{F}^2\\
	&= \min_{U_\mathrm{r}+iU_\mathrm{i}\in\mathrm{U}(m)}2\left\|\left(\frac{A_{11} + A_{22}}{2}\right) + i\left(\frac{A_{21} - A_{12}}{2}\right)-(U_\mathrm{r}+iU_\mathrm{i})\right\|_\mathrm{F}^2\\&+\frac14 \left\|AJ_{} - J_{}A\right\|_\mathrm{F}^2.
\end{align*}}Finally, it is well known that a matrix is a nearest unitary matrix if and only if it is a polar factor of $\left(\frac{A_{11} + A_{22}}{2}\right) + i\left(\frac{A_{21} - A_{12}}{2}\right)$. Moreover, the distance to the nearest unitary matrix is equal to $\|U\Sigma V^*-UV^*\|_\mathrm{F} = \|\Sigma-I\|_\mathrm{F}$. This concludes the proof.
\end{proof}
The proof of~\cref{thm:accuracy_repeated_improved} requires an explicit upper bound in terms of $A$ on the distance to a nearest ortho-symplectic matrix. We give this bound in~\cref{thm:bound_on_os}.
\begin{theorem}\label{thm:bound_on_os}
For every matrix $A\in\mathbb{R}^{2m\times 2m}$, we have
\begin{align*}
	\min_{R\in\mathrm{OSp}(2m)}\|A-R\|_\mathrm{F}^2\leq \frac{1}{4}\left(\|A^\top A - I_{}\|_\mathrm{F} + \|A^\top J_{} A - J_{}\|_\mathrm{F}\right)^2 + \frac{1}{4}\left\|AJ_{} - J_{}A\right\|_\mathrm{F}^2.
\end{align*}
\end{theorem}
\begin{proof}
Let us define $B = \left(\frac{A_{11} + A_{22}}{2}\right) + i\left(\frac{A_{21} - A_{12}}{2}\right) = U\Sigma V^*$. Then it follows that
\begin{align*}
2\|\Sigma-I_m\|_\mathrm{F}^2&\leq 2\|\Sigma^2-I_m\|_\mathrm{F}^2\\
&=2\|B^*B - I_m\|_\mathrm{F}^2\\
&=\|\frac{1}{4}(A + J_{}AJ_{}^\top)^\top (A + J_{}AJ_{}^\top)-I_{}\|_\mathrm{F}^2\\
&= \|\frac{1}{4}(A^\top A + J_{}A^\top AJ_{}^\top + A^\top J_{}AJ_{}^\top + J_{}A^\top J_{}^\top A  )-I_{}\|_\mathrm{F}^2\\
&\leq \frac{1}{16}\left(2\|A^\top A - I_{} \|_\mathrm{F}+2\|A^\top J_{} A J_{}^\top - I_{} \|_\mathrm{F}\right)^2\\
&= \frac{1}{4}\left(\|A^\top A - I_{} \|_\mathrm{F}+\|A^\top J_{} A - J_{} \|_\mathrm{F}\right)^2.
\end{align*}
By inserting the latter inequality in \eqref{eq:os_distance}, the claim follows.
\end{proof}
The following \cref{cor:bound_on_os} derives directly from~\cref{thm:bound_on_os}. It is used in the proof of~\cref{thm:accuracy_repeated_improved}
\begin{corollary}\label{cor:bound_on_os}
For every matrix $A\in\mathbb{R}^{2m\times 2m}$, we have
\begin{equation*}
	\min_{R\in\mathrm{OSp}(2m)}\|A-R\|_\mathrm{F}\leq \frac{1}{2}\left(\|A^\top A - I_{}\|_\mathrm{F} + \|A^\top J_{} A - J_{}\|_\mathrm{F}+\left\|AJ_{} - J_{}A\right\|_\mathrm{F}\right).
\end{equation*}
\end{corollary}
\section{Conclusion}
In this paper, we have designed a new Jacobi-like algorithm for computing the real Schur decomposition of a real normal matrix. We have shown in the numerical experiments of~\cref{sec:numerical_experiments} that \cref{alg:normalpaardekooper} is accurate and faster than the other Jacobi-like algorithms for normal matrices. An important contribution to this line of work would include an open-source multithreaded implementation of \cref{alg:normalpaardekooper}.
\appendix
\section{Improved perturbation theorem}
In this section we provide an improved version of \cite[Thm.~4.6]{mataignegallivan2025} imposing $R\in\mathrm{OSp}(2m)$ instead of $R\in\mathrm{O}(2m)$ for
		$$\min_{R\in\mathrm{OSp}(2m)}\|A(VR)-(VR)\left[\begin{smallmatrix}
		D_m&-\sigma I_m\\
		\sigma I_m&D_m
		\end{smallmatrix}\right]\|_\mathrm{F}, $$
where $D_m$ is a diagonal matrix. This improved version requires the results of \cref{thm:bound_on_os} and \cref{cor:bound_on_os}, obtained in \cref{sec:SSH_and_os}. Up to~\eqref{eq:slave2}, the proof of \cref{thm:accuracy_repeated_improved} is rigorously identical to that of \cite[Thm.~4.6]{mataignegallivan2025}. As previously, we use the short-hand notation $c_j\coloneq \cos(\theta_j)$ and $s_j = \sin(\theta_j)$.
\begin{theorem}\label{thm:accuracy_repeated_improved}
Let an RSD $A=QSQ^\top \in\mathbb{R}^{2p\times 2p}$ of a normal matrix with eigenvalues $\lambda_j e^{\pm i\theta_j}$, $\lambda_j>0,\ \theta_j\in(0,\pi)$ for $j=1,...,p$ and $\Omega \coloneq \mathrm{skew}(A)$. Assume there is $[j]\subseteq\{1,...,p\}$ such that $ \lambda_{j} s_{j}=\sigma$ for all $j\in[j]$. Set $[i]\coloneq\{1,...,p\}\setminus [j]$, $m\coloneq |[j]|$ and let $D_m$ be a diagonal matrix containing the real parts $\lambda_jc_j$ for all $j\in[j]$. If the matrix $V\in\mathrm{St}(2p,2m)$ satisfies $\Omega V-V\sigma\left[\begin{smallmatrix}0&-I_m\\
	I_m&0\end{smallmatrix}\right]=E$ with $\|E\|_\mathrm{F}\leq \tau \varepsilon_\mathrm{m} \|A\|_\mathrm{F}$, then we have
\begin{align*}
	&\min_{R\in\mathrm{OSp}(2m)}\|A(VR)-(VR)\left[\begin{smallmatrix}
		D_m&-\sigma I_m\\
		\sigma I_m&D_m
		\end{smallmatrix}\right]\|_\mathrm{F} \\
		&\leq\tau \varepsilon_\mathrm{m} \|A\|_\mathrm{F} \left(1 + \frac{\max\limits_{i\in[i],j\in[j]}|\lambda_i c_i-\lambda_j c_j|}{\min\limits_{i\in[i]}|\lambda_i s_i-\sigma|}+ \frac{\max\limits_{j_1,j_2\in[j]}|\lambda_{j_1} c_{j_1}-\lambda_{j_2} c_{j_2}|}{\sigma}\right)+\mathrm{o}(\varepsilon_\mathrm{m}).
\end{align*}
\end{theorem}
\begin{proof}	
	Define $X\coloneq Q^\top V$, then it follows that
	\begin{equation*}
	\mathrm{skew}(S)X-X\sigma \begin{bmatrix}0&-I_m\\
	I_m&0\end{bmatrix}=Q^\top E\eqcolon \widetilde{E}.
	\end{equation*}
There exist two permutations $P_\mathrm{l}\in\mathrm{O}(n),P_\mathrm{r}\in\mathrm{O}(2m) $ such that $P_\mathrm{l}^\top X P_\mathrm{r}=\left[\begin{smallmatrix}X_{[i],[j]}  \\X_{[j],[j]}\end{smallmatrix}\right]$. The blocks $X_{[i],[j]}$ and $X_{[j],[j]}$ comprise, respectively, all $2\times 2$ blocks $X_{ij}$ and $X_{jj}$ such that for all $i\in[i], j\in[j]$, we have
\begin{equation}\label{eq:sylvester_x}
\begin{bmatrix}
		0&-\lambda_i s_i\\
		\lambda_i s_i&0
	\end{bmatrix}X_{ij}-X_{ij}\begin{bmatrix}
		0&-\sigma\\
		\sigma&0
	\end{bmatrix}=(P_\mathrm{l}^\top \widetilde{E} P_\mathrm{r})_{ij}.
\end{equation} 
If it hold that $\widetilde{E}=0$, then we would have $X_{[i],[j]}=0 $ and $X_{[j],[j]}\in\mathrm{O}(2m)$. However, since $\widetilde{E}\neq 0$ and $X\in\mathrm{St}(n,2m)$, we have
\begin{align*}
 X_{[i],[j]}^\top X_{[i],[j]} +X_{[j],[j]}^\top X_{[j],[j]} = I_{2m} \Longrightarrow \| X_{[j],[j]}^\top X_{[j],[j]}-I_{2m}\|_\mathrm{F}\leq \| X_{[i],[j]}\|_\mathrm{F}^2,\\
 \text{by \eqref{eq:sylvester_x},}\quad \| X_{[i],[j]}\|_\mathrm{F}^2\leq \sum_{i\in[i], j\in[j]} \frac{1}{|\lambda_i s_i-\sigma|^2}\|(P_\mathrm{l}^\top \widetilde{E} P_\mathrm{r})_{ij}\|_\mathrm{F}^2\leq \frac{\|E\|_\mathrm{F}^2}{\min_{i\in[i]}|\lambda_i s_i-\sigma|^2}.
\end{align*}
The penultimate inequality follows from~\cite[Thm.~VII.2.8]{bhatia97}. The deviation of $X_{[j],[j]}$ from orthogonality can thus be quantified. In addition, by taking the SVD $ X_{[j],[j]}=U_X\Sigma_X V_X^\top$, we can write $\| X_{[j],[j]}^\top X_{[j],[j]}-I_{2m}\|_\mathrm{F}=\|\Sigma_X^2-I_{2m}\|_\mathrm{F}\geq \|\Sigma_X-I_{2m}\|_\mathrm{F}$, where the last inequality stands because $|\alpha^2-1|\geq |\alpha-1|$ for all $\alpha\geq 0$. Define $\left[\begin{smallmatrix}\widetilde{S}_{2(p-m)} &0 \\ 0 &\widetilde{S}_{2m}\end{smallmatrix}\right]\coloneq P_\mathrm{l}^\top S P_\mathrm{l}$, then for all $R\in \mathrm{OSp}(2m)$, we have 
\begin{align}
\nonumber
	\|AVR-VR\widetilde{S}_{2m}\|_\mathrm{F}&=\|SXR - XR\widetilde{S}_{2m}\|_\mathrm{F}\\
	\nonumber
	&\leq\|\mathrm{sym}(S)XR-XR\mathrm{sym}(\widetilde{S}_{2m})\|_\mathrm{F}+\|E\|_\mathrm{F}\\
	\label{eq:master_equation}
	&\leq \|\mathrm{sym}(\widetilde{S}_{2(p-m)})X_{[i],[j]}R-X_{[i],[j]}R\mathrm{sym}(\widetilde{S}_{2m})\|_\mathrm{F}\\
	\nonumber
	&\quad +\|\mathrm{sym}(\widetilde{S}_{2m})X_{[j],[j]}R-X_{[j],[j]}R\mathrm{sym}(\widetilde{S}_{2m})\|_\mathrm{F}+\|E\|_\mathrm{F}.
\end{align}
The first term of \eqref{eq:master_equation} can be bounded by
\begin{align}
\nonumber
	\|\mathrm{sym}(\widetilde{S}_{2(p-m)})X_{[i],[j]}R-X_{[i],[j]}R\mathrm{sym}(\widetilde{S}_{2m})\|_\mathrm{F}&\leq  \max_{i\in[i],j\in[j]}|\lambda_i c_i-\lambda_j c_j|\|X_{[i],[j]}R\|_\mathrm{F}\\
	\label{eq:slave1}
	&\leq \frac{\max_{i\in[i],j\in[j]}|\lambda_i c_i-\lambda_j c_j|}{\min_{i\in[i]}|\lambda_i s_i-\sigma|}\|E\|_\mathrm{F}.
\end{align}
Moreover, the second term of \eqref{eq:master_equation} can also be bounded by
\begin{align}
\nonumber
&\|\mathrm{sym}(\widetilde{S}_{2m})X_{[j],[j]}R-X_{[j],[j]}R\mathrm{sym}(\widetilde{S}_{2m})\|_\mathrm{F}\\
\nonumber
&= \|\mathrm{sym}(\widetilde{S}_{2m})(X_{[j],[j]}R-I_{2m})-(X_{[j],[j]}R-I_{2m})\mathrm{sym}(\widetilde{S}_{2m})\|_\mathrm{F}\\
\label{eq:slave2}
&\leq  \max_{j_1,j_2\in[j]}|\lambda_{j_1} c_{j_1}-\lambda_{j_2} c_{j_2}|\|X_{[j],[j]}R-I_{2m}\|_\mathrm{F}.
\end{align}
For $R^\star\in\argmin_{R\in \mathrm{OSp}(2m)} \|X_{[j],[j]}R-I_{2m}\|_\mathrm{F}$, by \cref{cor:bound_on_os}, we have 
\begin{align*}
\|X_{[j],[j]}R^\star-I_{2m}\|_\mathrm{F}&\leq \frac{1}{2}\big(\|X_{[j],[j]}^\top X_{[j],[j]} - I_{2m}\|_\mathrm{F} + \|X_{[j],[j]}^\top J_{2m} X_{[j],[j]} - J_{2m}\|_\mathrm{F}\\&+\left\|X_{[j],[j]}J_{2m} - J_{2m}X_{[j],[j]}\right\|_\mathrm{F}\big).  
\end{align*}
We know that $\|X_{[j],[j]}^\top X_{[j],[j]} - I_{2m}\|_\mathrm{F}\leq \frac{\|E\|_\mathrm{F}^2}{\min_{i\in[i]}|\lambda_i s_i-\sigma|^2}$ and, by~\eqref{eq:sylvester_x}, we have $\left\|X_{[j],[j]}J_{2m} - J_{2m}X_{[j],[j]}\right\|_\mathrm{F}\leq\frac{\|E\|_\mathrm{F}}{\sigma}$. Moreover, \eqref{eq:sylvester_x} yields
\begin{align*}
X_{[i],[j]}^\top \mathrm{skew}(\widetilde{S}_{2(p-m)})X_{[i],[j]} &+ \sigma X_{[j],[j]}^\top J_{2m} X_{[j],[j]} - \sigma J_{2m} = X^\top \widetilde{E}\\
\Longrightarrow \|X_{[j],[j]}^\top J_{2m} X_{[j],[j]} - J_{2m} \|_\mathrm{F}&\leq \frac{\|E\|_\mathrm{F}}{\sigma} + \frac{1}{\sigma}\|X_{[i],[j]}^\top \mathrm{skew}(\widetilde{S}_{2(p-m)})X_{[i],[j]}\|_\mathrm{F}\\
&\leq \frac{\|E\|_\mathrm{F}}{\sigma} + \frac{1}{\sigma}\|X_{[i],[j]}\|_\mathrm{F}^2 \|A\|_\mathrm{F}\\
&\leq \|E\|_\mathrm{F}\left(\frac{1}{\sigma} + \frac{\|E\|_\mathrm{F}\|A\|_\mathrm{F}}{\sigma \min_{i\in[i]}|\lambda_i s_i-\sigma|^2}\right).
\end{align*}
Acknowledging that $\|E\|_\mathrm{F}\leq \tau \varepsilon_\mathrm{m}\|A\|_\mathrm{F}$, we have
\begin{align*}
	&\min_{R\in\mathrm{OSp}(2m)}\|A(\widehat{V}R)-(\widehat{V}R)\left[\begin{smallmatrix}
		D_m&-\sigma I_m\\
		\sigma I_m&D_m
		\end{smallmatrix}\right]\|_\mathrm{F} \\
		&\leq \|E\|_\mathrm{F}\Biggl(1 + \frac{\max\limits_{i\in[i],j\in[j]}|\lambda_i c_i-\lambda_j c_j|}{\min\limits_{i\in[i]}|\lambda_i s_i-\sigma|} \\
		&+\max_{j_1,j_2\in[j]}|\lambda_{j_1} c_{j_1}-\lambda_{j_2} c_{j_2}|\biggl(\frac{\|E\|_\mathrm{F}}{2\min_{i\in[i]}|\lambda_i s_i-\sigma|^2}  + \frac{1}{\sigma} + \frac{\|E\|_\mathrm{F}\|A\|_\mathrm{F}}{2\sigma \min_{i\in[i]}|\lambda_i s_i-\sigma|^2}\biggr)\Biggr)\\
		&\leq\tau \varepsilon_\mathrm{m} \|A\|_\mathrm{F} \left(1 + \frac{\max\limits_{i\in[i],j\in[j]}|\lambda_i c_i-\lambda_j c_j|}{\min\limits_{i\in[i]}|\lambda_i s_i-\sigma|}+ \frac{\max\limits_{j_1,j_2\in[j]}|\lambda_{j_1} c_{j_1}-\lambda_{j_2} c_{j_2}|}{\sigma}\right)+\mathrm{o}(\varepsilon_\mathrm{m}).
\end{align*} 
\end{proof}
\section{Additional details on Jacobi-like algorithms}
\subsection{The Schur form of a $3\times 3$ skew-symmetric matrix}\label{sec:3x3skew} Paardekooper's method with $n$ odd requires the solution to the $3\times 3$ skew-symmetric EVP. Solving the problem requires two similarity transformations by Givens rotations. The first Givens rotation is obtained from the elements $\omega_{21}$ and $\omega_{31}$:
\begin{equation*}
\begin{bmatrix}
	1&0&0\\
	0&c_1&s_1\\
	0&-s_1&c_1
	\end{bmatrix}
	\underbrace{\begin{bmatrix}
		0&-\omega_{21}&-\omega_{31}\\
		\col{\omega_{21}}&0& - \omega_{32}\\
		\col{\omega_{31}}&\omega_{32}&0
	\end{bmatrix}}_{=\Omega}
	\begin{bmatrix}
	1&0&0\\
	0&c_1&-s_1\\
	0&s_1&c_1
	\end{bmatrix} = 
	\underbrace{\begin{bmatrix}
		0&-\omega_{21}'&0\\
		\col{\omega_{21}'}&0& - \omega_{32}'\\
		\col{0}&\omega_{32}'&0
	\end{bmatrix}}_{=\Omega^{(1)}},
\end{equation*}
where $c_1 = \frac{\omega_{21}}{\sqrt{\omega_{21}^2 + \omega_{31}^2}}$ and $s_1 = \frac{\omega_{31}}{\sqrt{\omega_{21}^2 + \omega_{31}^2}}$ define the Givens rotation. The second transformation is obtained from $-\omega_{21}'$ and $\omega_{32}'$ of $\Omega^{(1)}$:
\begin{equation*}
\small
 \begin{bmatrix}
	c_2&0&s_2\\
	0&1&0\\
	-s_2&0&c_2
	\end{bmatrix}\begin{bmatrix}
		0&\col{-\omega_{21}'}&0\\
		\omega_{21}'&0& - \omega_{32}'\\
		0&\col{\omega_{32}'}&0
	\end{bmatrix}\begin{bmatrix}
	c_2&0&-s_2\\
	0&1&0\\
	s_2&0&c_2
	\end{bmatrix} =\begin{bmatrix}
		0&\col{-\sigma}&0\\
		\sigma&0& 0\\
		0&\col{0}&0
	\end{bmatrix}, \quad \text{with,}\quad \sigma\geq 0,
\end{equation*}
where $c_2 = \frac{\omega_{21}'}{\sqrt{\omega_{21}'^2 + \omega_{32}'^2}}$ and $s_2 = \frac{-\omega_{32}'}{\sqrt{\omega_{21}'^2 + \omega_{32}'^2}}$.
These two steps give a simple, explicit solution to the $3\times 3$ skew-symmetric eigenvalue problem. Notice that $\sigma = \sqrt{\omega_{21}^2 +\omega_{31}^2+\omega_{32}^2 }$.
\subsection{The Jacobi-like algorithm for normal matrices from~\cite[Alg.~2]{BunseGerstner1993} and~\cite{ZhouBrent2003}}
For completeness of this paper, we provide a very simple pseudo-code of the Jacobi-like algorithms for normal matrices from~\cite[Alg.~2]{BunseGerstner1993} and~\cite{ZhouBrent2003} in \cref{alg:Zhou}. While~\cite{ZhouBrent2003} directly assumes the availability of a routine returning the RSD of an arbitrary $4\times 4$ matrix, \cite[Alg.~2]{BunseGerstner1993} is more involved. However, from a high-level point of view, \cite[Alg.~2]{BunseGerstner1993} is equivalent to~\cref{alg:Zhou}.
\begin{algorithm}
    \caption{Jacobi-like algorithm for normal matrices from~\cite[Alg.~2]{BunseGerstner1993} and~\cite{ZhouBrent2003}}
	\begin{algorithmic}
		\State \textbf{Input:} $A\in\mathbb{C}^{n\times n}$ normal, $Q\in\mathrm{Q}(n)$, $l \subseteq \{1,..., n\}$, $|l|\geq 4$ even and $\rho>0$.
		\While{$\mathrm{offschur}(A_{l, l})>\rho \|A\|_\mathrm{F}$}
			\For{$i = 1:2:|l|-3$}
				\For{$j = i+2:2:|l|-1$}
					\State Let $\mathtt{l} = \{l_i,l_{i+1},l_j,l_{j+1} \}$.
					\State Compute the Schur vectors $R$ of $A_{\mathtt{l}, \mathtt{l}}$ where $[R^\top A_{\mathtt{l},\mathtt{l}}R]_{3,2}= 0$.
					\State Update $A_{:,\mathtt{l}} \leftarrow A_{:,\mathtt{l}}R$ and $A_{\mathtt{l},:} \leftarrow R^\top A_{\mathtt{l},:}$.
					\State Update $Q_{:,\mathtt{l}} \leftarrow Q_{:,\mathtt{l}}R$.
				\EndFor
			\EndFor
		\EndWhile
		\Return $A$ and $Q$.
	\end{algorithmic}
	\label{alg:Zhou}
\end{algorithm}

\subsection{A Jacobi-like algorithm based on \texttt{RandDiag}}\label{sec:randdiag}
The randomized method \texttt{RandDiag} for computing the eigenvalue decomposition of a complex normal matrix $A\in\mathbb{C}^{n\times n}$ is proposed in~\cite{He2025}. It relies on the diagonalization of the Hermitian matrix $\widetilde{H} = \frac{\mu_1}{2}(A+A^*) + i \frac{\mu_2}{2}(A-A^*) = U\Lambda U^*$ where $\mu_1,\mu_2\sim \mathcal{N}(0,1)$. The authors show that $U$ diagonalizes $A$ with probability $1$ \cite[Thm.~1]{He2025}. An implicit Jacobi-like algorithm can be derived from \texttt{RandDiag}. Its pseudo-code is given in \cref{alg:randdiag}. It consists of applying Jacobi's algorithm for Hermitian matrices \emph{implicitly} to the Hermitian matrix by applying the rotations on $A$. 
\begin{algorithm}
    \caption{Implicit, cyclic, Jacobi-like \texttt{RandDiag} algorithm}
	\begin{algorithmic}
		\State \textbf{Input:} $A\in\mathbb{C}^{n\times n}$ normal, $Q\in\mathrm{O}(n)$ and $\rho>0$.
		\State Let $\mu_1,\mu_2\sim \mathcal{N}(0,1)$.
		\While{$\mathrm{offdiag}(\frac{\mu_1}{2}(A+A^*) + i \frac{\mu_2}{2}(A-A^*))>\rho\|A\|_\mathrm{F}$}
			\For{$i=1,...,n-1$}
				\For{$j = i+1,...,n$}
					\State Let $\mathtt{l} = \{i,j\}$.
					\State Compute the Hermitian block $H \coloneq \frac{\mu_1}{2}(A_{\mathtt{l},\mathtt{l}}+A_{\mathtt{l},\mathtt{l}}^*) + i \frac{\mu_2}{2}(A_{\mathtt{l},\mathtt{l}}-A_{\mathtt{l},\mathtt{l}}^*)$.
					\State Compute a Jacobi rotation $G$ such that $G^*HG$ is diagonal.
					\State Update $A_{:,\mathtt{l}} \leftarrow A_{:,\mathtt{l}}G$ and $A_{\mathtt{l},:} \leftarrow G^*A_{\mathtt{l},:}$.
					\State Update $Q_{:,\mathtt{l}} \leftarrow Q_{:,\mathtt{l}}G$.
				\EndFor
			\EndFor
		\EndWhile
		\Return $A$ and $Q$.
	\end{algorithmic}
	\label{alg:randdiag}
\end{algorithm}


\end{document}